\documentclass[review,onefignum,onetabnum]{siamart190516}

\usepackage{latexsym}
\usepackage{verbatim}
\usepackage{cite}
\usepackage[english]{babel}
\usepackage[latin1]{inputenc}
\usepackage{enumitem}
\usepackage{pmat}
\usepackage{amsmath}
\newtheorem{appxlem}{Lemma}[section]
\usepackage{url}
\usepackage{tikz}
\usetikzlibrary{tikzmark}
\usepackage{amsfonts}
\usepackage{amssymb}
\usepackage{amsbsy}

\usepackage{bm,palatino}

    \let\leq\leqslant
    \let\geq\geqslant
    \let\emptyset\varnothing
   
\makeatletter
\newsavebox\myboxA
\newsavebox\myboxB
\newlength\mylenA
\newcommand*\widebar[2][0.7]{%
    \sbox{\myboxA}{$\m@th#2$}%
    \setbox\myboxB\null
    \ht\myboxB=\ht\myboxA%
    \dp\myboxB=\dp\myboxA%
    \wd\myboxB=#1\wd\myboxA
    \sbox\myboxB{$\m@th\overline{\copy\myboxB}$}
    \setlength\mylenA{\the\wd\myboxA}
    \addtolength\mylenA{-\the\wd\myboxB}%
    \ifdim\wd\myboxB<\wd\myboxA%
       \rlap{\hskip 0.5\mylenA\usebox\myboxB}{\usebox\myboxA}%
    \else
        \hskip -0.5\mylenA\rlap{\usebox\myboxA}{\hskip 0.5\mylenA\usebox\myboxB}%
    \fi}
\makeatother

\DeclareMathOperator{\im}{im}

\DeclareMathOperator{\rank}{rank}

\let\leq\leqslant
\let\geq\geqslant
\let\emptyset\varnothing

\newcommand{\calF}{\ensuremath{\mathcal{F}}}

\newcommand{\calL}{\ensuremath{\mathcal{L}}}

\newcommand{\calN}{\ensuremath{\mathcal{N}}}

\newcommand{\calS}{\ensuremath{\mathcal{S}}}

\newcommand{\calV}{\ensuremath{\mathcal{V}}}

\newcommand{\calZ}{\ensuremath{\mathcal{Z}}}

\newcommand{\barZ}{\ensuremath{\bar{Z}}}

\newcommand{\bmat}{\begin{matrix}}
\newcommand{\emat}{\end{matrix}}
\newcommand{\bbm}{\begin{bmatrix}}
\newcommand{\ebm}{\end{bmatrix}}
\newcommand{\bbma}{\begin{bmatrix*}[r]}
\newcommand{\ebma}{\end{bmatrix*}}
\newcommand{\bpm}{\begin{pmatrix}}
\newcommand{\epm}{\end{pmatrix}}
\newcommand{\bvm}{\begin{vmatrix}}
\newcommand{\evm}{\end{vmatrix}}
\newcommand{\bse}{\begin{subequations}}
\newcommand{\ese}{\end{subequations}}
\newcommand{\beq}{\begin{equation}}
\newcommand{\eeq}{\end{equation}}
\newcommand{\ben}{\renewcommand{\labelenumi}{\arabic{enumi}.}
\renewcommand{\theenumi}{\arabic{enumi}}\begin{enumerate}}

\newcommand{\een}{\end{enumerate}}

\newcommand{\beni}{\renewcommand{\labelenumi}{\roman{enumi}.}
\renewcommand{\theenumi}{\roman{enumi}}\begin{enumerate}}

\newcommand{\eeni}{\end{enumerate}}

\newcommand{\bena}{\renewcommand{\labelenumi}{\alph{enumi}.}
\renewcommand{\theenumi}{\alph{enumi}}\begin{enumerate}}

\newcommand{\eena}{\end{enumerate}}

\newcommand{\bit}{\begin{itemize}}
\newcommand{\eit}{\end{itemize}}
\newcommand{\bthe}{\begin{theorem}}
\newcommand{\ethe}{\end{theorem}}
\newcommand{\blem}{\begin{lemma}}
\newcommand{\elem}{\end{lemma}}
\newcommand{\bprop}{\begin{proposition}}
\newcommand{\eprop}{\end{proposition}}
\newcommand{\bex}{\begin{example}}
\newcommand{\eex}{\end{example}}
\newcommand{\bas}{\begin{assumption}}
\newcommand{\eas}{\end{assumption}}
\newcommand{\bre}{\begin{remark}}
\newcommand{\ere}{\end{remark}}
\newcommand{\bcor}{\begin{corollary}}
\newcommand{\ecor}{\end{corollary}}
\newcommand{\bdfn}{\begin{definition}}
\newcommand{\edfn}{\end{definition}}
\newcommand{\bcon}{\begin{conjecture}}
\newcommand{\econ}{\end{conjecture}}

\newcommand{\ones}{\ensuremath{1\!\!1}}

\newcommand{\half}{\ensuremath{\frac{1}{2}}}
\newcommand{\inv}{\ensuremath{^{-1}}}

\newcommand{\set}[2]{\ensuremath{\{#1\mid #2\}}}

\newcommand{\norm}[1]{\ensuremath{\| #1 \|}}

\newcommand{\R}{\ensuremath{\mathbb R}}

\newcommand{\Z}{\ensuremath{\mathbb Z}}

\newcommand{\qand}{\quad\text{ and }\quad}

\newcommand{\bmu}{\bm{u}}
\newcommand{\bmx}{\bm{x}}
\newcommand{\bmv}{\bm{v}}
\newcommand{\bmw}{\bm{w}}
\newcommand{\bmy}{\bm{y}}

\newcommand{\bmf}{\bm{f}}

\newcounter{todocounter}
\usepackage[colorinlistoftodos]{todonotes}

\newtheorem{example}[theorem]{Example}
\newtheorem{remark}[theorem]{Remark}

\newcommand{\bpi}{\bm{\Pi}}
\newcommand{\mpi}{\bbm \Pi_{11}&\Pi_{12}\\\Pi_{21}&\Pi_{22}\ebm}
\renewcommand{\S}[1]{\mathbb{S}^{#1}}
\newcommand{\gi}{^\dagger}

\renewcommand{\set}[2]{\left\{#1:#2\right\}}

\newcommand{\schur}{\!\mid\!}

\newcommand{\pip}{\Pi_{22}}
\newcommand{\pis}{\Pi\schur\Pi_{22}}
\newcommand{\zs}{\calZ_r(\Pi)}

\DeclareMathOperator*{\argmin}{arg\,min}

\headers{Quadratic matrix inequalities}{H. J. van Waarde, M. K. Camlibel, J. Eising, and H. L. Trentelman}

\title{Quadratic matrix inequalities with applications to data-based control\thanks{Final submission on the 17th of February, 2023.}}
\author{Henk J. van Waarde\thanks{Bernoulli Institute for Mathematics, Computer Science and Artificial Intelligence, University of Groningen, The Netherlands (\email{h.j.van.waarde@rug.nl}, \email{m.k.camlibel@rug.nl},
\email{h.l.trentelman@rug.nl}).} \and M. Kanat Camlibel\footnotemark[2] \and Jaap Eising\thanks{Department of Mechanical and Aerospace Engineering, UC San Diego, USA
(\email{jeising@ucsd.edu}).} \and Harry L. Trentelman\footnotemark[2]}

\begin{document}

\nolinenumbers

\maketitle

\begin{abstract}
This paper studies several problems related to quadratic matrix inequalities (QMI's), i.e., inequalities in the Loewner order involving quadratic functions of matrix variables. In particular, we provide conditions under which the solution set of a QMI is nonempty, convex, bounded, or has nonempty interior. We also provide a parameterization of the solution set of a given QMI. In addition, we state results regarding the image of such sets under linear maps, which characterize a subset of ``structured" solutions to a QMI. Thereafter, we derive matrix versions of the classical S-lemma and Finsler's lemma, that provide conditions under which all solutions to one QMI also satisfy another QMI. The results will be compared to related work in the robust control literature, such as the full block S-procedure and Petersen's lemma, and it is demonstrated how existing results can be obtained from the results of this paper as special cases. Finally, we show how the various results for QMI's can be applied to the problem of data-driven stabilization. This problem involves finding a stabilizing feedback controller for an unknown dynamical system influenced by noise on the basis of a finite set of data. We provide general necessary and sufficient conditions for data-based quadratic stabilization. In addition, we demonstrate how to reduce the computational complexity of data-based stabilization by leveraging the aforementioned results. This involves separating the computation of the Lyapunov function and the controller, and also leads to explicit formulas for data-guided feedback gains.
\end{abstract}

\vspace{-5pt}

\vspace{-5pt}

\vspace{-10pt}\section{Introduction}

Designing control laws directly using measured data is becoming an increasingly prominent problem in systems and control \cite{Hjalmarsson1994,
Hjalmarsson1998,
Campi2002,
Willems2005,
Markovsky2008,
Campestrini2017,
Maupong2017,
Dai2018,
Baggio2019,
Coulson2019,
Umenberger2019,
Hewing2020,
Iannelli2020,
Ferizbegovic2020,
Steentjes2021b,
Eising2021,
Eising2022,
vanWaarde2022c}. This ``one shot" paradigm has several potential advantages over the two-step procedure of system identification combined with model-based controller synthesis. For example, direct data-driven control is simpler from a conceptual point of view and could thus be the preferred choice for practitioners. It has also been argued that stabilizing controllers can be obtained from data that lack the richness assumptions typically imposed for system identification \cite{vanWaarde2020}. Despite these advantages, direct data-driven control schemes have their own issues. For example, since system models can be regarded as ``condensed" representations of data, control methods that work directly with high-dimensional data matrices often suffer from a higher computational complexity than their model-based counterparts (see, e.g., \cite{Markovsky2021}). Therefore, the problem of reducing computational complexity of direct methods is very relevant. Another important problem is to obtain (non-asymptotic) guarantees of the performance of direct control schemes, working with a \emph{finite} number of noisy data samples, see e.g. \cite{Dean2019,Cohen2019}.

A fruitful and recent line of work that addresses the latter problem in the context of linear time-invariant (LTI) systems builds on the assumption that the matrix containing the noise samples is the solution to a \emph{quadratic matrix inequality} (QMI) \cite{DePersis2020,
Berberich2019c,
vanWaarde2022,
Koch2021,
vanWaarde2021,
Steentjes2021,
vanWaarde2022b,
Bisoffi2021,
Burohman2021,
Steentjes2022}. Such a QMI can capture, for example, energy bounds on the noise or bounds on the noise sample covariance matrix. These bounds give rise to a set of LTI systems that ``explain" the measured data, in the sense that each system in this set could have generated  the measurements \emph{for some} admissible noise sequence. Because of the assumption on the noise, the set of explaining systems can also be written as the set of solutions to a QMI. In this setting, direct data-driven control boils down to the robust control problem of finding a single controller that controls all systems explaining the data. Different control problems have been addressed in this manner, ranging from stabilization \cite{DePersis2020,Berberich2019c,vanWaarde2022,Bisoffi2021,Steentjes2022}, to $H_{2}$ and $H_{\infty}$ control \cite{Berberich2019c,vanWaarde2022,Steentjes2021}. Also data-based analysis problems such as dissipativity have been studied \cite{Koch2021,vanWaarde2022b}, as well as data-driven reduced order modeling \cite{Burohman2021}. The results are typically phrased in terms of data-dependent linear matrix inequalities (LMI's). A theorem that is of interest in this context is a matrix version of the classical \emph{S-lemma} \cite{Yakubovich1977}, that was developed in \cite{vanWaarde2022}. This matrix S-lemma provides a tractable LMI condition under which all solutions to one QMI also satisfy another QMI. By noting that Lyapunov and dissipation inequalities (with quadratic storage functions) can be regarded as QMI's in the system matrices, the matrix S-lemma can be used to establish conditions under which all systems explaining the data (i.e., solutions of a QMI) also satisfy the performance specifications (i.e., are also solutions to a second QMI).

Despite this insight, we believe that there are certain aspects of QMI's that are not yet completely understood. For example, the matrix S-lemma of \cite{vanWaarde2022} assumes a so-called \emph{generalized Slater condition}, which turns out to not be satisfied in the special case that the data are noise-free. Regardless, the paper \cite{vanWaarde2021} shows that the LMI condition for stabilization in \cite{vanWaarde2022} still holds in the exact data setting. The latter result for exact data was established by deriving a matrix version of \emph{Finsler's lemma}, that provides conditions under which all solutions to a \emph{quadratic matrix equality} satisfy a QMI. The paper \cite{Bisoffi2021} makes use of the so-called \emph{Petersen's lemma}, which enables the removal of the generalized Slater condition. Nonetheless, the results of \cite{Bisoffi2021} are only applicable to full rank data matrices, and to a specific noise model that does not capture all relevant cases (such as cross-covariance bounds, see \cite{Steentjes2022}). 

These observations suggest the existence of a general theory that encapsulates all of the above results. The purpose of this paper is to establish this theory. In fact, one of our main contributions is to derive matrix versions of the S-lemma and Finsler's lemma that do not rely on the generalized Slater condition. We study both the cases of strict and nonstrict QMI's. In the case of a strict inequality, we even show that the matrix Finsler's lemma and S-lemma can be unified to obtain a single, more general result. In addition to these highlights, an important contribution of the paper is to give a full account of properties of quadratic matrix inequalities. In particular, we provide conditions under which the set of solutions to a QMI is nonempty, bounded, convex, or has nonempty interior. We study both strict and nonstrict QMI's, as well as quadratic matrix equalities. Furthermore, we provide parameterizations of all solutions to a QMI. Thereafter, we study the image of solution sets of QMI's under linear maps. This means that we are interested in matrices of the form $ZW$, where $Z$ satisfies a QMI and $W$ is a given matrix. We will show that, under suitable conditions, the set of all these matrices coincides with the solution set of yet another QMI. This result is relevant for data-driven control whenever the noise is contained in a known subspace of the state-space \cite{vanWaarde2022}, or if certain prior knowledge of the system matrices is given \cite{Berberich2020d}. Finally, we will study applications of the various results on QMI's in the context of data-driven control. First, we will provide a general necessary and sufficient condition for data-based quadratic stabilization, using the unified matrix S-lemma and Finsler's lemma. Subsequently, we apply the results regarding linear maps to reduce the computational complexity of data-based stabilization. In particular, we completely separate the computation of the quadratic Lyapunov function and the controller, which leads to lower-dimensional LMI's. We also provide an explicit formula for a suitable stabilizing controller, given a (pre-computable) Lyapunov function. In the final part of the paper, we apply our results to resolve more intricate problems. First, we consider the nonlinear class of Lur'e systems, for which we derive stabilization results. After this, we consider the situation of linear systems with Gaussian noise. Our results will lead to conditions on the data which allow us to design a controller that has a given probability of stabilizing the true system.

Throughout the paper we will compare our work to related results in the literature. For instance, in Section~\ref{s:parameterization} we show the relation between \cite[Cor. 2.3.6]{Skelton1998} and the QMI parameterization result. In Section~\ref{s:comparison}, we compare the matrix S-lemma of this paper with a matrix S-lemma that can be derived from the full block S-procedure and the literature on LMI relaxations \cite{Scherer2001,Scherer2005}. The conclusion is that the latter literature is able to deal with more general uncertainty descriptions than the one in this paper. Nonetheless, in the context of uncertainty described by QMI's, the matrix S-lemma derived from \cite{Scherer2001,Scherer2005} follows from this paper. We also show how the matrix S-lemma of this paper improves that of \cite{vanWaarde2022} by removing the generalized Slater condition. Finally, we show that both the strict and nonstrict Petersen's lemma \cite{Petersen1987,Petersen1986}, that were recently applied to data-driven control \cite{Bisoffi2021}, can be obtained from our results. 

The outline is as follows. In Section~\ref{s:motivation} we motivate our study of QMI's with the application of data-based stabilization. Next, in Section~\ref{s:sets} we study basic properties of solution sets of QMI's, such as boundedness and convexity, and we establish the parameterization and results on images of solution sets of QMI's. Section~\ref{s:FinslerandSlemma} provides matrix versions of the S-lemma and Finsler's lemma. The results are applied to data-driven control in Section~\ref{sec:appl}. Finally, Section~\ref{s:conclusions} contains our conclusions. 

\subsection{Notation}
The \emph{Moore-Penrose pseudo-inverse} of a real matrix $A$ is denoted by $A^\dagger$. The set of real \emph{symmetric} $n \times n$ matrices is denoted by $\mathbb{S}^n$. In what follows, let $A \in \mathbb{S}^n$. The matrix $A$ is called \emph{positive semidefinite} if $x^\top A x \geq 0$ for all $x \in \mathbb{R}^n$ and \emph{positive definite} if $x^\top A x > 0$ for all nonzero $x\in \mathbb{R}^n$. This is denoted by $A \geq 0$ and $A > 0$, respectively. Negative semidefiniteness and negative definiteness is defined in a similar way, and denoted by $A \leq 0$ and $A < 0$, respectively. Throughout the paper, we will only consider definiteness conditions on symmetric matrices. As such, the notation $A \geq 0$ implies symmetry of $A$. In addition, the notation $A \leq B$ means that $A - B \leq 0$ (equivalently, $B-A \geq 0$), and $A < B$ means that $A-B<0$. We denote by $\lambda_{\mathrm{min}}(A)$ and $\lambda_{\mathrm{max}}(A)$ the smallest and largest eigenvalue of $A$, respectively. By \cite[Cor. 8.4.2]{Bernstein2009}, we have that $\lambda_{\mathrm{min}}(A)I\leq A\leq \lambda_{\mathrm{max}}(A)I$. If $A \geq 0$, there is exactly one positive semidefinite \emph{square root} of $A$, denoted by $A^{\frac{1}{2}}$, such that $A = A^{\frac{1}{2}}A^{\frac{1}{2}}$. 

\section{Motivation}
\label{s:motivation}
To demonstrate how quadratic matrix inequalities arise in the context of data-driven control, this section discusses data-based feedback stabilization. Consider input-state systems with unknown process noise of the form
\begin{equation} \label{systemprocessnoise}
\bmx(t+1) = A_s \bmx(t) + B_s \bmu(t) +\bmw(t),
\end{equation}
where $\bmx,\bmw \in \mathbb{R}^n$ and $\bmu \in \mathbb{R}^m$. The matrices $A_s$ and $B_s$ and the noise $\bmw(t)$ are unknown, but samples of the state and input are measured and collected in the matrices
\beq\label{e:def x umin}
X = \begin{bmatrix}
x(0) & x(1) & \cdots & x(T)
\end{bmatrix} \:\:\text{ and }\:\: U_- = \begin{bmatrix}
u(0) & u(1) & \cdots & u(T-1)
\end{bmatrix}.
\eeq
We also define $X_- = \begin{bmatrix}
x(0) & x(1) & \cdots & x(T-1)
\end{bmatrix}$, $X_+ = \begin{bmatrix}
x(1) & x(2) & \cdots & x(T)
\end{bmatrix}$ and $W_- = \begin{bmatrix}
w(0) & w(1) & \cdots & w(T-1)
\end{bmatrix}$. The matrix $W_-$ is not known, but assumed to satisfy a \emph{quadratic matrix inequality}, as explained next. Consider 
\[\Phi = \begin{bmatrix}
    \Phi_{11} & \Phi_{12} \\
    \Phi_{21} & \Phi_{22}
    \end{bmatrix} \in\S{n+T}, \textrm{ where } \Phi_{11} \in \mathbb{S}^n, \Phi_{12} \in \mathbb{R}^{n \times T}, \Phi_{21} \in \mathbb{R}^{T \times n}, \Phi_{22} \in \mathbb{S}^{T}. \]
We assume that the unknown noise $W_-$ satisfies
\begin{equation}
\label{asnoise}
\begin{bmatrix}
    I \\ W_-^\top 
    \end{bmatrix}^\top 
    \Phi
    \begin{bmatrix}
    I \\ W_-^\top 
    \end{bmatrix} \geq 0.
\end{equation}
The above noise model captures, for instance, 
\ben[label=(\roman*),ref=(\roman*),leftmargin=*]
\item\label{i:nm1} {\em energy bounds}: $\Phi_{22} =\! -I$ and $\Phi_{12} =\! 0$  imply $W_- W_-^\top \!=\! \sum_{t=0}^{T-1} w(t) w(t)^\top \!\leq \Phi_{11}$, which means that the energy of $w$ on the time interval $[0,T-1]$ is bounded by $\Phi_{11}$;
\item\label{i:nm2} {\em individual noise sample bounds}: The noise bound \eqref{asnoise} with $\Phi_{22} = -I$, $\Phi_{12} = 0$ and $\Phi_{11} = \epsilon T I$ 
is implied by $\|w(t)\|_2^2 \leq \epsilon\,\,\forall t$. The latter means that the individual noise samples at every time instant are bounded in norm;
\item\label{i:nm3} {\em sample covariance bounds:} The choices $\Phi_{22}\! = \frac{1}{T}(-I + \frac{1}{T}\ones \ones^\top)$ and $\Phi_{12} = 0$ with $\mu := \frac{1}{T} \sum_{t=0}^{T-1} w(t)$ lead to $\frac{1}{T} \sum_{t =0}^{T -1} (w(t) - \mu)  (w(t) - \mu)^\top= \frac{1}{T} W_-(I - \frac{1}{T}\ones \ones^\top)W_-^\top \leq \Phi_{11}$ where $\ones$ denotes the $T$-vector of ones. In other words, the sample covariance matrix of $w$ is bounded by $\Phi_{11}$;
\item\label{i:nm4} {\em bounded noise within a subspace:} Let $E \in \mathbb{R}^{n\times d}$ have full column rank and let $\hat{\Phi} \in \mathbb{S}^{d + T}$. Under suitable conditions, $W_- = E\hat{W}_-$ for some $\hat{W}_-$ satisfying 
$\begin{bmatrix} I \\ \hat{W}_-^\top \end{bmatrix}^\top \hat{\Phi} \begin{bmatrix} I \\ \hat{W}_-^\top \end{bmatrix} \geq 0$ if and only if $W_-$ and $\Phi := 
\begin{bmatrix} E & 0 \\ 0 & I \end{bmatrix} \hat{\Phi} \begin{bmatrix} E^\top & 0 \\ 0 & I \end{bmatrix}$  satisfy~\eqref{asnoise} (see Section~\ref{sec:appl}). This matrix $\Phi$ thus captures the situation that the noise is contained in the subspace $\im E$, and equal to $E\hat{W}_-$, where $\hat{W}_-$ again satisfies a quadratic matrix inequality;
\item\label{i:nm5} {\em exact measurements:} $\Phi_{11}=0$, $\Phi_{12}=0$, and $\Phi_{22}=-I$ leads to $W_- = 0$.
\een

\noindent
The set $\Sigma$ of all systems compatible with the data is given by \vspace{-6pt}
\begin{equation*}
\vspace{-6pt}  \Sigma = \{(A,B) \in \mathbb{R}^{n \times n} \times \mathbb{R}^{n \times m} \mid X_+ = A X_- + B U_- + W_- \text{ for some } 
 W_- \text{ satisfying }\eqref{asnoise}\}.
\end{equation*}
Since the measurements are obtained from the system \eqref{systemprocessnoise}, we clearly have that $(A_s,B_s)\in\Sigma$. Since on the basis of the given data we are not able to distinguish the true system $(A_s,B_s)$ from any other system in $\Sigma$, a stabilizing controller computed using only the data should necessarily stabilize all systems in $\Sigma$. This observation leads to the following notion of \emph{informativity} for quadratic stabilization.
\begin{definition}\label{def:quad stab}
The data $(U_-,X)$ are called \emph{informative for quadratic stabilization} if there exists a feedback gain $K$ and a matrix $P > 0$ such that for all $(A,B) \in \Sigma$:\vspace{-6pt}
\begin{equation}\label{lyapunovineq}
\vspace{-6pt}P - (A+BK) P (A+BK)^\top > 0.
\end{equation}
\end{definition}
Note that $P>0$ satisfies \eqref{lyapunovineq} if and only $Q:= P^{-1}$ satisfies $Q - (A +BK)^\top Q (A +BK) >0$, which expresses that $V(x) = x^T Q x$ is a Lyapunov function for the closed loop system $\bmx(t + 1) = (A + BK)\bmx(t)$.

Finding a controller that stabilizes all systems in a given set is a problem which has been studied under a number of different names in the literature. One of the most common of these is \textit{robust control}, where the set of systems takes the form of a nominal system and a set of perturbations. An overview of such methods pertaining to linear systems can be found in \cite{Green:2012}. More specifically an LMI-based approach to this problem can be found in \cite{Scherer2006}. The problem is also studied under the name of \textit{simultaneous control}, an overview of which can be found in \cite{Petersen:1987,Blondel:1994}. Apart from being studied on its own merits, results regarding sets of systems have been applied in the study of, for instance linear parameter varying systems \cite{Mohammadpour:2012} or switched systems \cite{Lin:07}.

Taking inspiration from the aforementioned works, we are interested in \emph{quadratic stabilization} in the sense that we ask for a \emph{common} Lyapunov matrix $P$ for all $(A,B) \in \Sigma$. This assumption is quite customary in data-driven \cite{DePersis2020,Berberich2019c,vanWaarde2022} and robust \cite{Scherer1999} control, and is made mainly because it leads to tractable solutions. This assumption is, however, often conservative in the sense that there may exist a \emph{parameter-dependent} Lyapunov function in cases that no common Lyapunov function exists. In \cite{vanWaarde2022e}, a first attempt was made to quantify the conservatism of common Lyapunov functions in the context of this paper. 

Recall that $(A,B)\in\Sigma$ if and only if $W_- = X_+ - A X_- -BU_-$, where $W_-$ is such that~\eqref{asnoise} holds. Combining this yields that $(A,B)\in\Sigma$ if and only if $(A,B)$ satisfies\vspace{-6pt}
\begin{equation}\label{ineqAB}
\vspace{-6pt}    \begin{bmatrix}
    I \\ A^\top \\ B^\top 
    \end{bmatrix}^\top 
    \begin{bmatrix}
    I & X_+ \\ 0 & -X_- \\ 0 & -U_-
    \end{bmatrix}
    \begin{bmatrix}
    \Phi_{11} & \Phi_{12} \\
    \Phi_{21} & \Phi_{22}
    \end{bmatrix}
    \begin{bmatrix}
    I & X_+ \\ 0 & -X_- \\ 0 & -U_-
    \end{bmatrix}^\top
    \begin{bmatrix}
    I \\ A^\top \\ B^\top 
    \end{bmatrix} \geq 0.
\end{equation}
In addition, if we fix $K$ and $P$ then \eqref{lyapunovineq} is yet another QMI in $(A,B)$:\vspace{-6pt}
\begin{equation}
\label{ineqABPK}
\vspace{-6pt}    \begin{bmatrix}
    I \\ A^\top \\ B^\top 
    \end{bmatrix}^\top
    \begin{bmatrix}
    P & 0 & 0 \\
    0 & -P & -PK^\top \\
    0 & -KP & -KPK^\top
    \end{bmatrix}
    \begin{bmatrix}
    I \\ A^\top \\ B^\top 
    \end{bmatrix} > 0.
\end{equation}
Therefore, finding conditions for informativity for quadratic stabilization amounts to finding conditions under which there exist $K$ and $P>0$ such that the quadratic matrix inequality \eqref{ineqABPK} holds for all $(A,B)$ satisfying the quadratic matrix inequality \eqref{ineqAB}. In other words, checking informativity requires us to check whether a given QMI implies another one. A generalization of the so-called \emph{S-lemma} \cite{Yakubovich1977,Polik2007} to matrix variables \cite[Thm. 13]{vanWaarde2022} leads to necessary and sufficient conditions for this implication in terms of an LMI.

\bprop
\label{theoremstab}
Assume that $\Phi_{22}<0$ and there exists $\bar{Z} \in \mathbb{R}^{(n+m) \times n}$ such that
\begin{equation}
    \label{matSlater}
\begin{bmatrix} I \\ \bar{Z} \end{bmatrix}^\top     \begin{bmatrix}
    I & X_+ \\ 0 & -X_- \\ 0 & -U_-
    \end{bmatrix}
    \begin{bmatrix}
    \Phi_{11} & \Phi_{12} \\
    \Phi_{21} & \Phi_{22}
    \end{bmatrix}
    \begin{bmatrix}
    I & X_+ \\ 0 & -X_- \\ 0 & -U_-
    \end{bmatrix}^\top\begin{bmatrix} I \\ \bar{Z} \end{bmatrix} > 0.
\end{equation}
Then the data $(U_-,X)$ are informative for quadratic stabilization if and only if there exists an $n\times n$ matrix $P > 0$, an $L \in \mathbb{R}^{m \times n}$ and scalars $\alpha \geq 0$ and $\beta > 0$ satisfying \vspace{-6pt}
\begin{equation}
\small\begin{bmatrix}
   P-\beta I & 0 & 0 & 0 \\
    0 & -P & -L^\top & 0 \\
    0 & -L & 0 & L \\
    0 & 0 & L^\top & P
    \end{bmatrix} - \alpha \begin{bmatrix}
    I & X_+ \\ 0 & -X_- \\ 0 & -U_- \\ 0 & 0
    \end{bmatrix}
    \begin{bmatrix}
    \Phi_{11} & \Phi_{12} \\
    \Phi_{21} & \Phi_{22}
    \end{bmatrix}
    \begin{bmatrix}
    I & X_+ \\ 0 & -X_- \\ 0 & -U_- \\ 0 & 0
    \end{bmatrix}^\top \geq 0.\normalsize \label{LMIstabmotivation}
\end{equation}
Moreover, if $P$ and $L$ satisfy \eqref{LMIstabmotivation} then $K := L P^{-1}$ is a stabilizing gain for all $(A,B) \in \Sigma$.
\eprop

Similar data-driven methods have also been provided in \cite{vanWaarde2022} for $H_2$ and $H_\infty$ controller design. However, these results cannot be applied to certain relevant noise models of the form \eqref{asnoise} because the assumptions $\Phi_{22} < 0$ and \eqref{matSlater} of Proposition~\ref{theoremstab} fail to hold. Indeed, the inequality $\Phi_{22} < 0$ is not satisfied by the noise model \ref{i:nm3} whereas \eqref{matSlater} is not satisfied by \ref{i:nm4} and \ref{i:nm5}. Nevertheless, it was proven in \cite[Thm. 2]{vanWaarde2021} that feasibility of \eqref{LMIstabmotivation} is still necessary and sufficient for informativity with exact data, i.e. noise model \ref{i:nm5}. For this noise model, $\Phi$ is negative semidefinite and hence \eqref{ineqAB} boils down to an equality. Therefore, to tackle the noiseless data case a matrix version of Finsler's lemma was developed in \cite[Thm. 1]{vanWaarde2021}. Further, \cite{Steentjes2022} considers a noise model capturing cross covariance bounds for which $\Phi_{22}$ is not necessarily negative definite and thus Proposition~\ref{theoremstab} cannot be directly applied. Under certain rank conditions on the data and for the noise model \ref{i:nm1}, the paper \cite{Bisoffi2021} used the Petersen's lemma to remove the condition \eqref{matSlater}. 

In this paper, we will provide an in-depth study of QMI's in a general framework which covers each of the noise models \ref{i:nm1}-\ref{i:nm5}. The results mentioned above can be recovered from the more general theory of this paper.

\section{Sets induced by quadratic matrix inequalities}
\label{s:sets}

In this section, we are interested in properties of sets defined in terms of quadratic matrix inequalities. We begin with sets of the form
$$
\calZ_{r}(\Pi):=\set{Z\in\R^{r\times q}}{\bbm I_q\\Z\ebm^\top\Pi\bbm I_q\\Z\ebm\geq 0},
$$
for $\Pi\in\S{q+r}$. The very first question one may ask is under what conditions on $\Pi$ the set $\calZ_{r}(\Pi)$ is nonempty. An immediate necessary condition is that $\Pi$ must have at least $q$ nonnegative eigenvalues. Clearly, this is not sufficient in general. What is necessary and sufficient for $Z\in\calZ_{r}(\Pi)$ is that the matrix\vspace{-5pt}
\beq\label{e: Phi Z matrix}
\small\bbm
\Pi_{11}&\Pi_{12}&-Z^\top\\\Pi_{21}&\Pi_{22}&I\\
-Z & I & 0
\ebm\normalsize, \quad \textrm{ where} \quad \Pi=\mpi
\eeq
has exactly $r$ negative eigenvalues (see \cite[Fact 8.15.28]{Bernstein2009}). From now on, whenever we partition a matrix $\Pi\in\S{q+r}$, we assume that $\Pi_{11}$ is $q \times q$ and $\Pi_{22}$ is $r \times r$. The condition on the eigenvalues of \eqref{e: Phi Z matrix} does not translate to an easily verifiable condition on $\Pi$ for nonemptiness of $\calZ_{r}(\Pi)$. Nevertheless, it leads to a noteworthy dualization result. To state this result, for given $\calS\subseteq\R^{r\times q}$, we define $\calS^\top:=\set{Z^\top}{Z\in\calS}$.
\vspace{-5pt}\bprop
Let $\Pi\in\S{q+r}$ be nonsingular and $\calZ_{r}(\Pi)\neq\emptyset$. Then, \vspace{-5pt} \[\vspace{-8pt}\big(\calZ_{r}(\Pi)\big)^\top=\calZ_{q}(\Pi^{\sharp}_{r,q}), \textrm{ where } \Pi^{\sharp}_{r,q}:=\bbm0&-I_r\\I_q&0\ebm\Pi\inv\bbm0&-I_q\\I_r&0\ebm.\] 
\eprop
A proof can be found in \cite[Lem. 4.9]{Scherer1999}. An alternative proof was given in \cite[Lem. 3]{vanWaarde2022b} based on the Schur complements of \eqref{e: Phi Z matrix} with respect to $\Pi$ and 
$
\bbm \Pi_{22}&I\\I & 0\ebm
$.

It turns out that for particular matrices $\Pi$, a Schur complement argument on the matrix $\Pi$ itself leads to a simple characterization of nonemptiness of the set $\calZ_{r}(\Pi)$. Specifically, suppose that $\Pi_{22}\leq 0$ and $\ker\Pi_{22}\subseteq\ker\Pi_{12}$. Since the latter condition is equivalent to $\Pi_{12}\Pi_{22}\Pi_{22}\gi=\Pi_{12}$, we have that\vspace{-5pt}
\begin{equation}\label{e:gen-schur}
\vspace{-5pt}\mpi=\bbm I_q & \Pi_{12}\Pi_{22}\gi\\0&I_r\ebm\bbm \Pi\schur\Pi_{22}& 0 \\0 & \Pi_{22}\ebm\bbm I_q & 0\\\Pi_{22}\gi\Pi_{21} & I_r\ebm
\end{equation}
where $\Pi\schur\Pi_{22}:=\Pi_{11}-\Pi_{12}\Pi_{22}\gi\Pi_{21}$ is the (generalized) Schur complement of $\Pi$ with respect to $\Pi_{22}$. This results in \vspace{-5pt}
\begin{equation}\label{e:expression for ztpiz}
\vspace{-5pt}\begin{bmatrix}
I_q \\ Z
\end{bmatrix}^\top
\mpi
\begin{bmatrix}
I_q \\ Z
\end{bmatrix}
=
\Pi\schur\Pi_{22}+(Z+\Pi_{22}\gi\Pi_{21})^\top\Pi_{22}(Z+\Pi_{22}\gi\Pi_{21})
\end{equation}
because $(\Pi_{22}^\dagger)^\top = (\Pi_{22}^\top)^\dagger = \Pi_{22}^\dagger$ by symmetry of $\Pi_{22}$ and, since $\Pi_{22} \leq 0$,\vspace{-5pt}
\begin{equation}\label{e:obs gen schur}
\vspace{-5pt}\Pi\schur\Pi_{22}=
\begin{bmatrix}
I_q \\ -\Pi_{22}\gi\Pi_{21}
\end{bmatrix}^\top
\mpi
\begin{bmatrix}
I_q \\ -\Pi_{22}\gi\Pi_{21}
\end{bmatrix}
\geq \begin{bmatrix}
I_q \\ Z
\end{bmatrix}^\top
\mpi
\begin{bmatrix}
I_q \\ Z
\end{bmatrix}
\end{equation}
for any $Z \in \mathbb{R}^{r\times q}$. An immediate consequence of this inequality is that $\calZ_{r}(\Pi)\neq\emptyset$ if and only if $\Pi\schur\Pi_{22}\geq 0$. Motivated by this observation, we define the set\vspace{-5pt}
\begin{equation}
\label{Piqr}
\vspace{-5pt}\bpi_{q,r}=\set{\mpi\in\S{q+r}}{\Pi_{22}\leq 0, \Pi\schur\Pi_{22} \geq 0 \text{ and }\ker\Pi_{22}\subseteq\ker\Pi_{12}}.
\end{equation}
All the noise models mentioned after \eqref{asnoise} can be captured by elements of $\bpi_{n,T}$.

Next, for $\Pi\in\S{q+r}$ we will investigate properties of $\calZ_{r}(\Pi)$ and the sets\vspace{-5pt}
\[\vspace{-5pt}\calZ_r^+(\Pi) := \left\{\! Z \in \mathbb{R}^{r\times q} : 
\begin{bmatrix}I_q \\ Z\end{bmatrix}^\top\!\! \Pi \begin{bmatrix}I_q \\ Z\end{bmatrix}\!\! > 0 \right\}, \calZ_{r}^0(\Pi):=\set{\!Z\in\R^{r\times q}}{\bbm I_q\\Z\ebm^\top\!\!\Pi\bbm I_q\\Z\ebm \!\!= 0}\!\!,\]

\subsection{Basic properties}
\label{s:basicproperties}

In the following theorem, we study nonemptiness, convexity and boundedness of the sets induced by QMI's. 
\bthe\label{t:basic props}
Let $\Pi\in\bpi_{q,r}$. Then, $\calZ_r(\Pi)$ 
\begin{enumerate}[label=\emph{(\alph*)}, ref=\alph*]
\item\label{t:basic props.1} is nonempty and convex.
\item\label{t:basic props.2} is bounded if and only if $\Pi_{22}<0$.
\item\label{t:basic props.3} has nonempty interior if and only if $\Pi_{22}=0$ or $\Pi\schur\Pi_{22}>0$ .
\end{enumerate}
Further,  
\begin{enumerate}[resume,label=\emph{(\alph*)}, ref=\alph*]
\item\label{t:basic props.4} $\calZ^+_{r}(\Pi)$ is nonempty if and only if $\Pi\schur\Pi_{22}>0$.
\item\label{t:basic props.5} $\calZ^0_{r}(\Pi)$ is nonempty if and only if $\rank\Pi_{22}\geq\rank \Pi\schur\Pi_{22}$.
\end{enumerate}
\ethe

\begin{proof}
\eqref{t:basic props.1}: Since $\Pi\schur\Pi_{22}\geq 0$, it follows from \eqref{e:obs gen schur} that $-\Pi_{22}\gi\Pi_{21}\in\calZ_r(\Pi)$. This proves nonemptiness whereas convexity readily follows from $\Pi_{22}\leq 0$. 

\eqref{t:basic props.2}: We first prove the `if' part. Let $Z\in\calZ_r(\Pi)$. Then, it follows from \eqref{e:expression for ztpiz} that $
(Z+\Pi_{22}\gi\Pi_{21})^\top(-\Pi_{22})(Z+\Pi_{22}\gi\Pi_{21})\leq \Pi\schur\Pi_{22}$. This leads to $
\lambda_{\mathrm{min}}(-\Pi_{22})(Z+\Pi_{22}\gi\Pi_{21})^\top(Z+\Pi_{22}\gi\Pi_{21}) \leq \lambda_{\mathrm{max}}(\Pi\schur\Pi_{22})I$. 
Since $-\Pi_{22}>0$ and $\Pi\schur\Pi_{22}\geq 0$, we see that $\norm{Z+\Pi_{22}\gi\Pi_{21}}\leq \alpha$ for some $\alpha\geq 0$. Hence, $\calZ_r(\Pi)$ is bounded.

For the `only if' part, let $Z\in\calZ_r(\Pi)$ and let $\xi$ be such that $\Pi_{22}\xi=0$. Since $\Pi \in \bpi_{q,r}$, we see that $Z+\alpha\xi\xi^\top\in\calZ_r(\Pi)$ for any $\alpha \in \mathbb{R}$. Since $\calZ_r(\Pi)$ is bounded, this implies $\xi=0$. This proves that $\Pi_{22}$ is nonsingular. Thus $\Pi_{22}\leq 0$ implies $\Pi_{22}<0$.

\eqref{t:basic props.3}: For the `if' part, let $\Delta\in\R^{r\times q}$ be such that $\norm{\Delta}\leq 1$. For all $\epsilon>0$, we have\vspace{-5pt}
\begin{align*}
\pis+\epsilon^2\Delta^\top\pip\Delta 
&\geq\lambda_{\mathrm{min}}(\pis)I+\epsilon^2\lambda_{\mathrm{min}}(\pip)\Delta^\top\Delta \\
&\geq\lambda_{\mathrm{min}}(\pis)I+\epsilon^2\lambda_{\mathrm{min}}(\pip)I  \end{align*}
where the last inequality follows from the facts that $\pip\leq 0$ and $\Delta^\top\Delta\leq I$. If $\pip=0$, then the right hand side is nonnegative for any $\epsilon$ since $\pis\geq 0$. If $\pis>0$, then the right hand side is nonnegative for all sufficiently small $\epsilon>0$. Therefore, there exists $\epsilon>0$ such that $\pis+\epsilon^2\Delta^\top\pip\Delta\geq 0$ for all $\Delta$ with $\norm{\Delta}\leq 1$. Now, take $Z_0=-\Pi_{22}\gi\Pi_{21}$ and note that this implies \vspace{-5pt}
\[\vspace{-5pt}\Pi\schur\Pi_{22}+(Z_0+\epsilon\Delta+\Pi_{22}\gi\Pi_{21})^\top\Pi_{22}(Z_0+\epsilon\Delta+\Pi_{22}\gi\Pi_{21})=\pis+\epsilon^2\Delta^\top\pip\Delta\geq 0\]
for all $\Delta$ with $\norm{\Delta}\leq1$. Then, it follows from \eqref{e:expression for ztpiz} that $Z_0+\epsilon\Delta\in\zs$ for all $\Delta$ with $\norm{\Delta}\leq1$. This means that the set $\zs$ has nonempty interior.

For the `only if' part, suppose that $Z_0$ is in the interior of $\zs$. This means that there exists $\epsilon>0$ such that $Z_0+\epsilon\Delta\in\zs$ for all $\Delta$ with $\norm{\Delta}\leq1$. By \eqref{e:expression for ztpiz},\vspace{-5pt} 
\begin{equation}\label{e:nonempty interior 2}
\vspace{-5pt}\Pi\schur\Pi_{22}+(Z_0+\epsilon\Delta+\Pi_{22}\gi\Pi_{21})^\top\Pi_{22}(Z_0+\epsilon\Delta+\Pi_{22}\gi\Pi_{21})\geq 0.
\end{equation}
Suppose that $\xi\in\R^q$ is such that $(\pis)\xi=0$. Since $\pip\leq 0$, \eqref{e:nonempty interior 2} yields the equation $\Pi_{22}(Z_0+\epsilon\Delta+\Pi_{22}\gi\Pi_{21})\xi=0
$ for all $\Delta$ with $\norm{\Delta}\leq1$. By taking $\Delta=0$, we see that $\Pi_{22}(Z_0+\Pi_{22}\gi\Pi_{21})\xi=0$. Therefore, $\Pi_{22}\Delta\xi=0$ for all $\Delta$ with $\norm{\Delta}\leq1$. In particular, consider $\Delta = \zeta\xi^\top$. Then, we conclude that $\Pi_{22}\zeta\xi^\top\xi=0$ for all $\zeta\in\R^r$. Therefore, either $\pip=0$ or $\xi=0$. Equivalently, either $\pip=0$ or $\pis>0$.

\eqref{t:basic props.4}: For the `if' part, suppose that $\Pi\schur\Pi_{22}> 0$. Then, it follows from \eqref{e:obs gen schur} that $-\Pi_{22}\gi\Pi_{21}\in\calZ^+_r(\Pi)$. Thus, $\calZ^+_r(\Pi)$ is nonempty. For the `only if' part, suppose that $\calZ^+_r(\Pi)$ is nonempty. Let $Z\in\calZ^+_r(\Pi)$. Then, \eqref{e:obs gen schur} implies that
$\Pi\schur\Pi_{22}>0$.

\eqref{t:basic props.5}: For the `only if' part, suppose that $\calZ^0_r(\Pi)$ is nonempty. Let $Z\in\calZ^0_r(\Pi)$. Then, it follows from \eqref{e:expression for ztpiz} that $\Pi\schur\Pi_{22}=-(Z+\Pi_{22}\gi\Pi_{21})^\top\Pi_{22}(Z+\Pi_{22}\gi\Pi_{21})$. Since the rank of a product of matrices is less than or equal to the ranks of individual matrices, we see that $\rank\Pi_{22}\geq\rank (\Pi\schur\Pi_{22})$. For the `if' part, suppose that $\rank\Pi_{22}\geq\rank (\Pi\schur\Pi_{22})$. Let $U_1\Sigma_1U_1^\top$ and $U_2\Sigma_2U_2^\top$ be singular value decompositions of $\Pi\schur\Pi_{22}$ and $-\Pi_{22}$, respectively. Then, $\rank\Sigma_2\geq\rank\Sigma_1$. Hence, there exists a diagonal matrix $D\geq 0$ such that $\Sigma_1=D\Sigma_2$. Take $\barZ=-\Pi_{22}\gi\Pi_{21}+U_2D^{\half}U_1^T$. Note that $(\barZ+\Pi_{22}\gi\Pi_{21})^\top\Pi_{22}(\barZ+\Pi_{22}\gi\Pi_{21})=-U_1\Sigma_1U_1^\top=-\Pi\schur\Pi_{22}$. Consequently, it follows from \eqref{e:expression for ztpiz} that $\barZ\in\calZ^0_r(\Pi)$ and $\calZ^0_r(\Pi)$ is nonempty.
\end{proof}

\subsection{Parameterization of $\calZ_{r}(\Pi)$ and $\calZ^+_{r}(\Pi)$}
\label{s:parameterization}

It turns out that one can parameterize all solutions of a given QMI associated with $\Pi \in \bpi_{q,r}$.
\begin{theorem}\label{thm:para new}
Let $\Pi\in\bpi_{q,r}$. The following statements hold:
\ben[label=\emph{(\alph*)}, ref=\alph*,leftmargin=*]
\item\label{thm:para new.1} $Z\in\calZ_{r}(\Pi)$ if and only if there exist $S,T\in\R^{r\times q}$ with $S^\top S\leq I$ such that \vspace{-5pt} 
\begin{equation}\label{eq:Z equal}\vspace{-5pt}  Z=-\Pi_{22}\gi\Pi_{21}+\big((-\Pi_{22})\gi\big)^\half S (\Pi\schur\Pi_{22})^\half+(I-\Pi_{22}\gi\Pi_{22}) T.\end{equation}

\item\label{thm:para new.2} Assume that $\calZ_r^+(\Pi)$ is nonempty, equivalently, $\Pi\schur \Pi_{22} > 0$. Then, $Z\in\calZ^+_{r}(\Pi)$ if and only if $Z$ is of the form \eqref{eq:Z equal} for some $S,T\in\R^{r\times q}$ with $S^\top S < I$.
\een
\end{theorem} 

We note that, in the special case that $\Pi_{22} < 0$, Theorem~\ref{thm:para new}.\eqref{thm:para new.2} also follows from \cite[Corollary 2.3.6]{Skelton1998} by taking $A = \Pi_{12}\Pi_{22}^{-1}$, $B = I$, $Q = \Pi \schur \Pi_{22}$, $R = - \Pi_{22}$ and $X = Z^\top$ in that result. Here we provide a proof in the case that $\Pi_{22}$ is not necessarily negative definite and hence $\calZ_r(\Pi)$ is not necessarily bounded.

\begin{proof} 
We first prove \eqref{thm:para new.2}. From \eqref{e:expression for ztpiz} we have that $Z\in\calZ^+_{r}(\Pi)$ if and only if\vspace{-5pt} 
\begin{equation}\label{eq:equivstrict} \vspace{-5pt} (Z+\Pi_{22}\gi\Pi_{21})^\top(-\Pi_{22})(Z+\Pi_{22}\gi\Pi_{21}) < \Pi\schur\Pi_{22}. \end{equation}
By Lemma~\ref{l:A A <= B B 2}.\ref{item<I}, we then have that $Z\in\calZ^+_{r}(\Pi)$ if and only if there exists a matrix $S$ such that $S^\top S < I$ and $
(-\Pi_{22})^\half (Z+\Pi_{22}\gi\Pi_{21}) = S (\Pi\schur\Pi_{22})^\half$. Using the fact that $\ker (-\Pi_{22})^\half = \ker \Pi_{22}$, and by exploiting Lemma~\ref{lem: image-exist}, we see that this is equivalent to $Z+\Pi_{22}\gi\Pi_{21} = \big((-\Pi_{22})^\half \big)\gi S (\Pi\schur\Pi_{22})^\half + (I-\Pi_{22}\gi\Pi_{22})T$ for some $T\in\R^{r\times q}$. This proves \eqref{thm:para new.2}. The proof of \eqref{thm:para new.1} follows the same arguments but instead invokes Lemma~\ref{l:A A <= B B 2}.\ref{item<=I}.
\end{proof}

\subsection{Image of $\calZ_{r}(\Pi)$ and $\calZ^+_{r}(\Pi)$ under linear maps}
\label{s:projections}

Let $W\in\R^{q\times p}$. For $\calS\subseteq\R^{r\times q}$, we define $\calS W:=\set{SW}{S\in\calS}$. Also, for $\Pi \in \mathbb{S}^{q+r}$ we define \vspace{-5pt} 
\begin{equation}\label{e:PiW}
\vspace{-5pt} \Pi_W:=\bbm W^\top & 0\\0 & I_r\ebm\Pi\bbm W & 0\\0& I_r\ebm=\bbm W^\top\Pi_{11}W&W^\top\Pi_{12}\\\Pi_{21}W&\Pi_{22}\ebm\in\S{p+r}.
\end{equation}
Note that $\Pi_W\in\bpi_{p,r}$ if $\Pi\in\bpi_{q,r}$. Next, we will study the sets $\calZ_{r}(\Pi)$ and $\calZ_{r}(\Pi_W)$.
\bthe\label{t:projectionfcr}
Let $\Pi\in\bpi_{q,r}$ and $W\in\R^{q\times p}$. We have that $\calZ_{r}(\Pi)W \subseteq \calZ_{r}(\Pi_W)$. Assume, in addition, that either $W$ has full column rank or $\Pi_{22}$ is nonsingular.
Then, $\calZ_{r}(\Pi)W=\calZ_{r}(\Pi_W)$.
\ethe

\begin{proof} First we prove that $\calZ_{r}(\Pi)W\subseteq\calZ_{r}(\Pi_W)$. Let $Z'\in\zs W$. Then, $Z'=ZW$ where $Z\in\zs$. Note that \vspace{-5pt} 
\[\vspace{-5pt} \bbm I_q\\Z\ebm^\top\Pi\bbm I_q\\Z\ebm\geq 0, \quad \textrm{implies that } W^\top \bbm I_q\\Z\ebm^\top\Pi\bbm I_q\\Z\ebm W=\bbm I_p\\Z'\ebm^\top\Pi_W\bbm I_p\\Z'\ebm\geq 0.\]
This means that $Z'\in\calZ_{r}(\Pi_W)$ and hence $\calZ_{r}(\Pi)W\subseteq\calZ_{r}(\Pi_W)$.

Now, we assume that at least one of the conditions on $W$ and $\Pi_{22}$ hold. We claim that $\calZ_{r}(\Pi_W)\subseteq\calZ_{r}(\Pi)W$. Let $Z'\in\calZ_{r}(\Pi_W)$. Note that $\Pi_W\schur\Pi_{22}=W^\top(\Pi\schur\Pi_{22})W$. From \eqref{e:PiW} and Theorem~\ref{thm:para new}.\eqref{thm:para new.1}, we see that\vspace{-5pt} 
\begin{equation}\label{e:z prime}
\vspace{-5pt} Z'=-\Pi_{22}\gi\Pi_{21}W+\big((-\Pi_{22}\big)\gi\big)^\half S\big(W^\top(\Pi\schur\Pi_{22})W\big)^\half+(I-\Pi_{22}\gi\Pi_{22})V
\end{equation}
where $V,S\in\R^{r\times p}$ with $S^\top S\leq I_p$. Since $\big(W^\top(\Pi\schur\Pi_{22})W\big)^\half\big(W^\top(\Pi\schur\Pi_{22})W\big)^\half = W^\top (\Pi\schur\Pi_{22})^\half (\Pi\schur\Pi_{22})^\half W$, we have that $\big(W^\top(\Pi\schur\Pi_{22})W\big)^\half=T(\Pi\schur\Pi_{22})^\half W$ where $T\in\R^{p\times q}$ is such that $T^\top T\leq I_q$ due to Lemma~\ref{l:A A <= B B 2}. If $W$ has full column rank then \eqref{e:z prime} results in $Z'=ZW$ where \vspace{-5pt} 
\begin{equation}\label{e:from z prime to z}
\vspace{-5pt} Z := -\Pi_{22}\gi\Pi_{21}+\big((-\Pi_{22}\big)\gi\big)^\half ST(\Pi\schur\Pi_{22})^\half+(I-\Pi_{22}\gi\Pi_{22})V(W^\top W)\inv W^\top.
\end{equation}
On the other hand, if $\Pi_{22}$ is nonsingular then $I - \Pi_{22}^\dagger \Pi_{22} = 0$ and $Z' = ZW$ with\vspace{-5pt} 
\begin{equation}\label{e:Zcase2}
\vspace{-5pt} Z := -\Pi_{22}\inv\Pi_{21}+\big(-\Pi_{22}\inv\big)^\half ST(\Pi\schur\Pi_{22})^\half.
\end{equation}
In either of these two cases, we observe that $T^\top S^\top S T\leq T^\top T\leq I_q$. Therefore, Theorem~\ref{thm:para new}.\eqref{thm:para new.1} implies that $Z\in\zs$. Consequently, we see that $Z'=ZW$ for some $Z \in\zs$ and thus $\calZ_{r}(\Pi_W)\subseteq\calZ_{r}(\Pi)W$. This proves the theorem.
\end{proof}

A similar result holds for the sets $\calZ^+_{r}(\Pi)$ and $\calZ^+_{r}(\Pi_W)$, as shown next.

\bthe\label{t:projectionfcrstrict}
Let $\Pi\in\bpi_{q,r}$ and $W\in\R^{q\times p}$. Assume that $W$ has full column rank and $\calZ_r^+(\Pi)$ is nonempty. Then, $\calZ^+_{r}(\Pi)W=\calZ^+_{r}(\Pi_W)$.
\ethe

The proof of Theorem~\ref{t:projectionfcrstrict} is similar to that of Theorem~\ref{t:projectionfcr}, but applies Theorem~\ref{thm:para new}.\eqref{thm:para new.2} instead of Theorem~\ref{thm:para new}.\eqref{thm:para new.1}. The following two corollaries follow from Theorems~\ref{t:projectionfcr} and \ref{t:projectionfcrstrict} and provide conditions under which there exists a ``structured" matrix in $\mathcal{Z}_r(\Pi)$ (respectively, $\mathcal{Z}_r^+(\Pi)$) that satisfies a linear equation.\vspace{-5pt} 
\begin{corollary}
\label{cor:elimination}
Let $\Pi \in \S{q+r}$ with $\Pi_{22} \leq 0$ and $\ker \Pi_{22} \subseteq \ker \Pi_{21}$, $W \in \mathbb{R}^{q\times p}$, and $Y \in \mathbb{R}^{r \times p}$. Suppose that either $W$ has full column rank or $\Pi_{22}$ is nonsingular. Then there exists a $Z \in \mathcal{Z}_r(\Pi)$ such that $Z W = Y$ if and only if $\Pi \in \bpi_{q,r}$ and $Y \in \mathcal{Z}_r(\Pi_W)$.
\end{corollary} 
\vspace{-5pt} \begin{proof}
To prove the `if' statement, suppose that $\Pi \in \bpi_{q,r}$ and $Y \in \mathcal{Z}_r(\Pi_W)$. By Theorem~\ref{t:projectionfcr} there exists a $Z\in \mathcal{Z}_r(\Pi)$ such that $ZW = Y$. 

To prove the `only if' statement, suppose that there exists a $Z \in \mathcal{Z}_r(\Pi)$ satisfying $ZW = Y$. Therefore, $\mathcal{Z}_r(\Pi)$ is nonempty and \eqref{e:expression for ztpiz} implies that $\Pi \schur \Pi_{22} \geq 0$. Consequently, $\Pi \in \bpi_{q,r}$. Finally, $Y \in \mathcal{Z}_r(\Pi_W)$ follows directly from multiplying the defining quadratic matrix inequality from left by $W^\top$ and right by $W$.
\end{proof}
\vspace{-5pt} \begin{corollary}
\label{cor:strictelimination}
Let $\Pi \in \S{q+r}$ with $\Pi_{22} \leq 0$ and $\ker \Pi_{22} \subseteq \ker \Pi_{21}$. Consider $W \in \mathbb{R}^{q\times p}$ and $Y \in \mathbb{R}^{r \times p}$. Assume that $W$ has full column rank. Then there exists a matrix $Z \in \mathcal{Z}_r^+(\Pi)$ satisfying $Z W = Y$ if and only if $\Pi\schur\Pi_{22} > 0$ and $Y \in \mathcal{Z}_r^+(\Pi_W)$.
\end{corollary} 

The proof of Corollary~\ref{cor:strictelimination} follows the same lines as that of Corollary~\ref{cor:elimination}, but applies Theorem~\ref{t:projectionfcrstrict} rather than Theorem~\ref{t:projectionfcr}. It is therefore omitted. Corollary~\ref{cor:strictelimination} is intimately related to the so-called \emph{elimination lemma} \cite{Hermersson1999,Scherer2001}. 
In fact, in the case that $\Pi$ is nonsingular and has $r$ negative and $q$ positive eigenvalues, Corollary~\ref{cor:strictelimination} can also be obtained from \cite[Lem. A.2]{Scherer2001} by taking $P = -\Pi$, $A = I$, $B = W^\perp$ and $C = YW^\dagger$ where $W^\perp \in \mathbb{R}^{(q-p)\times q}$ is any full row rank matrix such that $W^\perp W = 0$.

\section{Matrix versions of Finsler's lemma and Yakubovich's S-lemma}
\label{s:FinslerandSlemma}

In this section we deal with the question under what conditions all solutions to one quadratic matrix inequality also satisfy another QMI. We aim at finding necessary and sufficient conditions for the inclusion $\calZ_r(N) \subseteq \calZ_r(M)$, where $M,N\in \mathbb{S}^{q+r}$. We will also consider $Z_r^0(N)$ instead of $Z_r(N)$, and $Z_r^+(M)$ replacing $Z_r(M)$. This leads to non-strict and strict versions of Yakubovich's S-lemma and Finsler's lemma. 

\subsection{Recap of standard S-lemma and Finsler's lemma}

For future reference, we will start with a brief recap of ``standard" (vector-valued) S-lemmas and Finsler's lemma. The idea behind all of these results is that certain implications involving quadratic inequalities and equalities can be reformulated as linear matrix inequalities. The following statement is the S-lemma for non-strict inequalities, which was first proven by Yakubovich in \cite{Yakubovich1977}, see also the survey paper \cite{Polik2007}.

\begin{lemma}[S-lemma]
\label{l:standardS-lemma}
Let $M,N \in \mathbb{S}^n$ and suppose that $N$ has at least one positive eigenvalue. Then $x^\top M x \geq 0$ for all $x\in \mathbb{R}^n$ satisfying
$x^\top N x \geq 0$ if and only if there exists a real number $\alpha \geq 0$ such that $M-\alpha N \geq 0$. 
\end{lemma}

\vspace{-5pt}Next, we recall a variant with a strict inequality on $x^\top M x$, see \cite{Yakubovich1977} and \cite[p. 24]{Boyd1994}.

\begin{lemma}[Strict S-lemma]
\label{l:strictS-lemma}
Let $M,N \in \mathbb{S}^n$ and suppose that $N$ has at least one positive eigenvalue. Then $x^\top M x > 0$ for all nonzero $x\in \mathbb{R}^n$ satisfying
$x^\top N x \geq 0$ if and only if there exists a real number $\alpha \geq 0$ such that $M-\alpha N > 0$. 
\end{lemma}

The strict S-lemma has also been generalized \cite{Iwasaki2000} from the set $\set{\alpha N}{\alpha \geq 0}$ of multiplier matrices to a more general set satisfying the so-called \emph{losslessness} property. Finally, we recall Finsler's lemma \cite{Finsler1936} which involves an \emph{equality} $x^\top N x = 0$. We state the result for a strict inequality on $x^\top Mx$. We note that also a non-strict version of the result exists (see e.g., \cite{Zi-zong2010}), but this will not be used in this paper. 

\begin{lemma}[Finsler's lemma]
\label{l:standardFinsler'slemma}
Let $M,N \in \mathbb{S}^n$. Then $x^\top M x > 0$ for all nonzero $x\in \mathbb{R}^n$ satisfying
$x^\top N x = 0$ if and only if there exists a real number $\alpha \in \mathbb{R}$ such that $M-\alpha N > 0$. 
\end{lemma}

\subsection{Reduction of the matrix case to the vector case}

Throughout this section, we will consider matrices $M,N \in  \mathbb{S}^{q+r}$ partitioned as \vspace{-5pt}
\begin{equation}
\label{partitionMN}
\vspace{-5pt}M = \begin{bmatrix}
M_{11} & M_{12} \\ M_{21} & M_{22}
\end{bmatrix} \text{ and } N = \begin{bmatrix}
N_{11} & N_{12} \\ N_{21} & N_{22}
\end{bmatrix}.
\end{equation}

We will provide conditions under which the inclusion $\calZ_r(N) \subseteq \calZ_r(M)$ is equivalent to the \emph{vector-valued} implication $x^\top N x \geq 0 \implies x^\top M x \geq 0$. This will provide an important building block in obtaining matrix versions of the S-lemma. To proceed, we will need the following lemma.
 
\begin{lemma}
\label{l:xS}
Let $S \in \mathbb{S}^n$, $S \geq 0$. Given a nonzero vector $x \in \mathbb{R}^n$, there exists a matrix $\bar{X} \in \mathbb{R}^{n \times (n-1)}$ such that $x^\top S \bar{X} = 0$ and $\begin{bmatrix}
x & \bar{X}
\end{bmatrix}$ is nonsingular. 
\end{lemma}

\begin{proof}
If $x^\top S = 0$ the statement is immediate. Thus, assume that $x^\top S \neq 0$. Let $\bar{X} \in \mathbb{R}^{n \times (n-1)}$ be a matrix whose columns form a basis for $\ker x^\top S$. If $\begin{bmatrix}
x & \bar{X}
\end{bmatrix}$ is singular, then $x\in \im \bar{X}$ and, hence, $x^\top S x = 0$. However, since $S$ is symmetric and positive semidefinite, this implies that $x^\top S = 0$. This yields a contradiction, and we conclude that $\begin{bmatrix}
x & \bar{X}
\end{bmatrix}$ is nonsingular. This proves the lemma. 
\end{proof}

Next, we state the following instrumental lemma that is an extension of \cite[Lemma A.2]{Scherer2005} to the set of matrices $\bpi_{q,r}$. Note that we do not require the matrix $N$ to satisfy $N_{22} < 0$ and $N_{11} - N_{12} N_{22}^{-1} N_{21} > 0$.

\begin{lemma}
\label{l:vectormatrix}
Let $N \in \bpi_{q,r}$. Let $x \in \mathbb{R}^{q}$ and $y \in \mathbb{R}^r$ be vectors, with $x$ nonzero, such that $\begin{bmatrix}
x \\ y
\end{bmatrix}^\top N \begin{bmatrix}
x \\ y
\end{bmatrix} \geq 0.
$ Then there exists a matrix $Z \in \calZ_r(N)$ such that $y=Zx$.
\end{lemma}
\begin{proof}
Since $x$ is nonzero and $N \schur N_{22} \geq 0$, we conclude from Lemma~\ref{l:xS} that there exists a matrix $\bar{X} \in \mathbb{R}^{q \times (q-1)}$ such that $x^\top (N\schur N_{22}) \bar{X} = 0$ and $\begin{bmatrix}
x & \bar{X}
\end{bmatrix}$ is nonsingular. Define the matrix $\bar{Y}:= - N_{22}^\dagger N_{21} \bar{X}$. Note that \vspace{-5pt}
\[
\vspace{-5pt}\begin{bmatrix}
N_{11} & N_{12} \\
N_{21} & N_{22}
\end{bmatrix} = \begin{bmatrix}
I & N_{12} N_{22}^\dagger \\0 & I 
\end{bmatrix}
\begin{bmatrix}
N \schur N_{22} & 0 \\ 0 & N_{22}
\end{bmatrix}
\begin{bmatrix}
I & 0 \\ N_{22}^\dagger N_{21} & I
\end{bmatrix}.\]
Therefore, we have\vspace{-5pt}
\[\vspace{-5pt}\begin{bmatrix}
x \\ y 
\end{bmatrix}^\top N \begin{bmatrix}
\bar{X} \\ \bar{Y}
\end{bmatrix}
 = 0 \qand
\begin{bmatrix}
\bar{X} \\ \bar{Y}
\end{bmatrix}^\top N \begin{bmatrix}
\bar{X} \\ \bar{Y}
\end{bmatrix} = \bar{X}^\top (N\schur N_{22})\bar{X} \geq 0
\]
since $N \schur N_{22} \geq 0$. The latter two results imply that\vspace{-5pt}
\[\vspace{-5pt}\begin{bmatrix}
x & \bar{X} \\ y & \bar{Y}
\end{bmatrix}^\top N \begin{bmatrix}
x & \bar{X} \\ y & \bar{Y}
\end{bmatrix} = \small \begin{bmatrix}
\begin{bmatrix}
x\\y
\end{bmatrix}^\top N \begin{bmatrix}
x\\y
\end{bmatrix} & 0 \\ 0 & \begin{bmatrix}
\bar{X}\\\bar{Y}
\end{bmatrix}^\top N \begin{bmatrix}
\bar{X}\\ \bar{Y}
\end{bmatrix}
\end{bmatrix} \normalsize \geq 0.\]
Recall that $\begin{bmatrix}
x & \bar{X}
\end{bmatrix}$ is nonsingular. Thus, the matrix $Z:=\begin{bmatrix}
y & \bar{Y}
\end{bmatrix}\begin{bmatrix}
x & \bar{X}
\end{bmatrix}^{-1}
$ is a member of $\calZ_r(N)$. In addition, note that $Zx = y$. This proves the lemma. 
\end{proof}

The following theorem provides conditions under which $\calZ_r(N) \subseteq \calZ_r(M)$ and $\calZ_r(N) \subseteq \calZ_r^+(M)$ are equivalent to their respective vector-valued implications.

\begin{theorem}
\label{t:equivmatrixvector}
Let $M,N \in \mathbb{S}^{q+r}$ with $N \in \bpi_{q,r}$. 
\begin{enumerate}[label=\emph{(\alph*)}, ref=\alph*]
\item\label{i:ns} Assume that $N$ has at least one positive eigenvalue. Then the following two statements are equivalent:
\begin{enumerate}[label=\emph{(\roman*)}, ref=\roman*]
\item\label{i:ns1} $\calZ_r(N) \subseteq \calZ_r(M)$, 
\item\label{i:ns2} $z^\top M z \geq 0$ for all $z \in \mathbb{R}^{q+r}$ satisfying $z^\top N z \geq 0$. 
\end{enumerate}
\item\label{i:s} Assume that $N_{22} < 0$. Then the following two statements are equivalent:
\begin{enumerate}[label=\emph{(\roman*)}, ref=\roman*]
\item\label{i:s1} $\calZ_r(N) \subseteq \calZ_r^+(M)$, 
\item\label{i:s2} $z^\top M z > 0$ for all nonzero $z \in \mathbb{R}^{q+r}$ satisfying $z^\top N z \geq 0$. 
\end{enumerate}
\end{enumerate}

\end{theorem}

\begin{proof}
We first prove that \eqref{i:ns}.\eqref{i:ns1} implies \eqref{i:ns}.\eqref{i:ns2}. Assume that \eqref{i:ns}.\eqref{i:ns1} holds but, on the contrary, \eqref{i:ns}.\eqref{i:ns2} does not hold. This implies that there exist vectors $x \in \mathbb{R}^q$ and $y \in \mathbb{R}^r$, not both zero, such that\vspace{-5pt}
\begin{equation}
\label{ineqsxyMN}
\vspace{-5pt}\begin{bmatrix}
x \\ y
\end{bmatrix}^\top N \begin{bmatrix}
x \\ y
\end{bmatrix} \geq 0 \text{ and } \begin{bmatrix}
x \\ y
\end{bmatrix}^\top M \begin{bmatrix}
x \\ y
\end{bmatrix} < 0.
\end{equation} 
We claim that there exists a pair $(x,y)$ satisfying \eqref{ineqsxyMN} with $x \neq 0$. 

To see this, suppose that $x = 0$ and $y$ satisfy \eqref{ineqsxyMN}. We will use these vectors to construct a new pair $(\tilde{x},\tilde{y})$ satisfying \eqref{ineqsxyMN} with $\tilde{x} \neq 0$. By the hypothesis that $N_{22} \leq 0$, this implies that $N_{22} y = 0$. In addition, since $\ker N_{22} \subseteq \ker N_{12}$ we have that $N \begin{bmatrix}
0 \\ y
\end{bmatrix}
=0$. Let $
\begin{bmatrix}
\bar{x}^\top & \bar{y}^\top
\end{bmatrix}^\top 
$ be an eigenvector of $N$ corresponding to a positive eigenvalue $\lambda$. Note that $\bar{x} \neq 0$ because $N_{22} \leq 0$. The previous implies that\vspace{-5pt}
\[\vspace{-5pt}\left(\begin{bmatrix}
0 \\ y
\end{bmatrix} + \epsilon \begin{bmatrix}
\bar{x} \\ \bar{y}
\end{bmatrix}\right)^\top N \left(\begin{bmatrix}
0 \\ y
\end{bmatrix} + \epsilon \begin{bmatrix}
\bar{x} \\ \bar{y}
\end{bmatrix}\right) = \epsilon^2 \lambda \begin{bmatrix}
\bar{x} \\ \bar{y}
\end{bmatrix}^\top \begin{bmatrix}
\bar{x} \\ \bar{y}
\end{bmatrix} \geq 0\]
for any $\epsilon \in \mathbb{R}$. In addition,\vspace{-5pt}
\begin{equation}
\label{xbarybarineq}
\vspace{-5pt}\left(\begin{bmatrix}
0 \\ y
\end{bmatrix} + \epsilon \begin{bmatrix}
\bar{x} \\ \bar{y}
\end{bmatrix}\right)^\top M \left(\begin{bmatrix}
0 \\ y
\end{bmatrix} + \epsilon \begin{bmatrix}
\bar{x} \\ \bar{y}
\end{bmatrix}\right) < 0
\end{equation}
if $\epsilon$ is sufficiently small. Therefore, for sufficiently small $\epsilon \neq 0$, the pair $(\epsilon \bar{x}, y + \epsilon \bar{y})$ satisfies \eqref{ineqsxyMN}. As $\bar{x} \neq 0$, there exists a pair $(\tilde{x},\tilde{y}):=(\epsilon \bar{x}, y + \epsilon \bar{y})$ satisfying \eqref{ineqsxyMN} with $\tilde{x} \neq 0$. By Lemma~\ref{l:vectormatrix} there exists a matrix $Z \in \calZ_r(N)$ such that $\tilde{y}=Z\tilde{x}$. By \eqref{xbarybarineq} we see that
$
\tilde{x}^\top \begin{bmatrix}
I \\ Z
\end{bmatrix}^\top M \begin{bmatrix}
I \\ Z
\end{bmatrix} \tilde{x} < 0,
$ that is, $Z \not\in \calZ_r(M)$. This, however, contradicts the assumption that $\calZ_r(N) \subseteq \calZ_r(M)$. Therefore, we conclude that \eqref{i:ns}.\eqref{i:ns2} holds. 

Next, we prove that \eqref{i:s}.\eqref{i:s1} implies \eqref{i:s}.\eqref{i:s2}. Therefore, assume that \eqref{i:s}.\eqref{i:s1} holds but, on the contrary, \eqref{i:s}.\eqref{i:s2} does not hold. This implies that there exist vectors $x \in \mathbb{R}^q$ and $y \in \mathbb{R}^r$, not both zero, such that 
$\begin{bmatrix}
x \\ y
\end{bmatrix}^\top N \begin{bmatrix}
x \\ y
\end{bmatrix} \geq 0$ and $\begin{bmatrix}
x \\ y
\end{bmatrix}^\top M \begin{bmatrix}
x \\ y
\end{bmatrix} \leq 0$. 
This implies that $x \neq 0$. Indeed, if $x = 0$ then also $y = 0$ by the hypothesis that $N_{22} < 0$. Thus, by Lemma~\ref{l:vectormatrix} there exists a matrix $Z \in \calZ_r(N)$ such that $y=Zx$. Combining the previous yields $x^\top \begin{bmatrix}
I \\ Z
\end{bmatrix}^\top
M \begin{bmatrix}
I \\ Z
\end{bmatrix}x \leq 0$,
that is, $Z \not\in \calZ_r^+(M)$. This contradicts the hypothesis that $\calZ_r(N) \subseteq \calZ_r^+(M)$. This shows that \eqref{i:s}.\eqref{i:s2} holds. 

Next, we prove that \eqref{i:ns}.\eqref{i:ns2} implies \eqref{i:ns}.\eqref{i:ns1}. Suppose that $Z \in \calZ_r(N)$. Then: \vspace{-5pt}
\[\vspace{-5pt}y^\top \begin{bmatrix}
I \\ Z
\end{bmatrix}^\top N \begin{bmatrix}
I \\ Z
\end{bmatrix} y \geq 0, \:\: \text{ and thus } \:\: y^\top \begin{bmatrix}
I \\ Z
\end{bmatrix}^\top M \begin{bmatrix}
I \\ Z
\end{bmatrix} y \geq 0\]
for all $y \in \mathbb{R}^q$. In other words, $Z \in \calZ_r(M)$. The proof that \eqref{i:s}.\eqref{i:s2} implies \eqref{i:s}.\eqref{i:s1} is analogous and therefore omitted. This proves the theorem.
\end{proof}

\subsection{Matrix versions of the S-lemma and Finsler's lemma for non-strict inequalities}

In the following theorem we apply the results of the previous section to establish a matrix version of the S-lemma.

\begin{theorem}[Matrix S-lemma]
\label{t:nonstrictS-lemma}
Let $M,N \in \mathbb{S}^{q+r}$. If there exists a real $\alpha \geq 0$ such that $M - \alpha N \geq 0$, then $\calZ_r(N) \subseteq \calZ_r(M)$. Next, assume that $N \in \bpi_{q,r}$ and $N$ has at least one positive eigenvalue. Then $\calZ_r(N) \subseteq \calZ_r(M)$ if and only if there exists a real $\alpha \geq 0$ such that $M-\alpha N \geq 0$. 
\end{theorem}

\begin{proof}
The `if' statements are obvious. We thus focus on proving the `only if' part of the second statement. Assume that $\calZ_r(N) \subseteq \calZ_r(M)$. By Theorem~\ref{t:equivmatrixvector}, $x^\top M x \geq 0$ for all $x\in \mathbb{R}^{q+r}$ satisfying $x^\top N x \geq 0$.  Finally, by Lemma~\ref{l:standardS-lemma}, we conclude that there exists a scalar $\alpha \geq 0$ such that $M-\alpha N \geq 0$.
\end{proof}

Similar to the `standard' S-lemma (Lemma~\ref{l:standardS-lemma}), we note that the matrix S-lemma requires $N$ to have at least one positive eigenvalue, an assumption known as the \emph{Slater condition}. It turns out, however, that under additional assumptions on $M$ and $N$, we can state a theorem analogous to Theorem~\ref{t:nonstrictS-lemma} in the case where $N \in \bpi_{q,r}$ has no positive eigenvalues, equivalently, $N \schur N_{22} = 0$. In this special case, $\calZ_r(N) = \calZ_r^0(N)$ which leads to a matrix version of Finsler's lemma. 

\begin{theorem}[Matrix Finsler's lemma]
\label{t:matFinsler}
Let $M,N \in  \mathbb{S}^{q+r}$. If there exists $\alpha \in \mathbb{R}$ such that $M - \alpha N \geq 0$ then $\calZ_r^0(N) \subseteq \calZ_r(M)$. Next, define $\Theta \in \mathbb{S}^q$ by\vspace{-5pt}
\begin{equation}\label{eq:def Theta}\vspace{-5pt} \Theta := 
\begin{bmatrix}
I \\ -N_{22}^\dagger N_{21}
\end{bmatrix}^\top M \begin{bmatrix}
I \\ -N_{22}^\dagger N_{21} 
\end{bmatrix}.
\end{equation}
Assume that $M,N \in \bpi_{q,r}$, $N\schur N_{22} = 0$, and $\ker \Theta\subseteq \ker M\schur M_{22}$.
Then $\calZ_r^0(N) \subseteq \calZ_r(M)$ if and only if there exists $\alpha \geq 0$ such that $M - \alpha N \geq 0$.
\end{theorem}

\begin{proof}
The `if' statements are obvious. 
Now, assume that  $\calZ_r^0(N) \subseteq \calZ_r(M)$. Let $Z \in \calZ_r^0(N)$, $\xi \in \ker N_{22}$, and $\eta \in \mathbb{R}^q$ be a nonzero vector.
By hypothesis, we have \vspace{-5pt}
\begin{equation}
\label{Zhat}
\vspace{-5pt}Z+\gamma \xi\eta^\top \in \calZ_r(M)
\end{equation}
for all $\gamma \in \mathbb{R}$. Recall that $M \in \bpi_{q,r}$ and therefore $M_{22} \leq 0$. This implies that $M_{22} \xi = 0$, for otherwise there exists a sufficiently large $\gamma \in \mathbb{R}$ that violates \eqref{Zhat}. We have thus proven that $\ker N_{22} \subseteq \ker M_{22}$. Let $T := \begin{bmatrix}
I & 0 \\ -N_{22}^\dagger N_{21} & I
\end{bmatrix}$, then \vspace{-5pt}
\[\vspace{-5pt}T^\top N T = \begin{bmatrix}
0 & 0 \\ 0 & N_{22}
\end{bmatrix} \text{ and } 
T^\top M T = \begin{bmatrix}
\Theta & M_{12} -N_{12} N_{22}^\dagger M_{22} \\
M_{21}-M_{22} N_{22}^\dagger N_{21} & M_{22}
\end{bmatrix}.\]
This yields\vspace{-5pt}
\begin{equation}
\label{matrixTMaNT}
\vspace{-5pt}T^\top (M-\alpha N) T = \begin{bmatrix}
\Theta & M_{12} -N_{12} N_{22}^\dagger M_{22} \\
M_{21}-M_{22} N_{22}^\dagger N_{21} & M_{22} - \alpha N_{22}
\end{bmatrix}.
\end{equation}
Next, note that $\ker M_{22} \subseteq \ker M_{12}$ implies the following two statements:\vspace{-5pt}
\begin{subequations}
	\begin{align}
	\label{rewriteTheta}
	\Theta = M\schur M_{22} + (M_{22}^\dagger M_{21} - N_{22}^\dagger N_{21})^\top M_{22} (M_{22}^\dagger M_{21} - N_{22}^\dagger N_{21})\\
	\label{eqM22} M_{22} (M_{22}^\dagger M_{21} - N_{22}^\dagger N_{21}) = M_{21} -M_{22}N_{22}^\dagger N_{21}.
	\end{align}
\end{subequations}
Note that $-N_{22}^\dagger N_{21} \in \calZ_r^0(N) \subseteq \calZ_r(M)$ and thus, $\Theta \geq 0$. 
Therefore, since $M_{22} \leq 0$, \eqref{rewriteTheta} and \eqref{eqM22} imply $\ker(M\schur M_{22}) = \ker \Theta \cap \ker(M_{21} - M_{22} N_{22}^\dagger N_{21})$. Therefore, by the hypothesis that $\ker \Theta \subseteq \ker (M\schur M_{22})$ we must have $\ker \Theta = \ker (M\schur M_{22})$, and it follows that $\ker \Theta = \ker(M_{21} - M_{22} N_{22}^\dagger N_{21})$. Consequently, by \eqref{matrixTMaNT} and $\Theta \geq 0$, we see that $T^\top (M-\alpha N) T \geq 0$ if and only if 
\begin{equation}
\label{ineqM22N22}
M_{22} - \alpha N_{22} - (M_{21}-M_{22} N_{22}^\dagger N_{21}) \: \Theta^\dagger (M_{12} -N_{12} N_{22}^\dagger M_{22}) \geq 0. 
\end{equation}

Since $N_{22} \leq 0$ and $\ker N_{22} \subseteq \ker M_{22} \subseteq \ker M_{12}$, we conclude that there exists a sufficiently large $\alpha \geq 0$ such that \eqref{ineqM22N22} holds. This implies that there exists an $\alpha \geq 0$ such that $M-\alpha N \geq 0$. This proves the theorem. 
\end{proof}

The assumption on the matrix $\Theta$ is required in the sense that Theorem~\ref{t:matFinsler} is, in general, not valid without it. We illustrate this as follows.
\begin{example}
Suppose that $ N=\bbm -1 & 1\\1 &-1\ebm \: \text{ and } M = \begin{bmatrix}
1 & 0 \\ 0 & -1 \end{bmatrix}$. Note that $M,N \in \bpi_{1,1}$ and $N \schur N_{22} = 0$. In this case, $\Theta = 0$ and $M \schur M_{22} = 1$ so the assumption on $\Theta$ of Theorem~\ref{t:matFinsler} does not hold. In addition, we see that $\calZ_1^0(N) = \{1\} \subseteq \calZ_1(M)$. Nonetheless, there does not exist an $\alpha \geq 0$ such that $M-\alpha N \geq 0$. 
\end{example}

\subsection{Matrix versions of the S-lemma and Finsler's lemma for a strict inequality}

Subsequently, we consider {\em strict} versions of the above theorems. We focus on the case that the inequality involving $M$ is strict, while the inequality on $N$ is nonstrict. We note that a matrix S-lemma with two strict inequalities was considered in \cite{Szabo2013}. The following theorem provides a strict matrix S-lemma in case $N_{22}$ is negative definite.

\begin{theorem}[Strict matrix S-lemma]
\label{t:strictS-lemmaN22}
Let $M,N \in \mathbb{S}^{q+r}$. If there exists a real $\alpha \geq 0$ such that $M - \alpha N > 0$, then $\calZ_r(N) \subseteq \calZ_r^+(M)$. Next, assume that $N \in \bpi_{q,r}$ and $N_{22} < 0$. Then $\calZ_r(N) \subseteq \calZ_r^+(M)$ if and only if there exists a real $\alpha \geq 0$ such that $M-\alpha N > 0$. 
\end{theorem}

\begin{proof}
The `if' parts are clear. Therefore, we focus on proving the `only if' part of the second statement. Suppose that $\calZ_r(N) \subseteq \calZ_r^+(M)$. By Theorem~\ref{t:equivmatrixvector}, we have that $x^\top M x > 0$ for all nonzero $x \in \mathbb{R}^{q+r}$ satisfying $x^\top N x \geq 0$. We now distinguish two cases. First suppose that $N$ has at least one positive eigenvalue. Then, by Lemma~\ref{l:strictS-lemma}, there exists a real $\alpha \geq 0$ such that $M-\alpha N > 0$. Next, suppose that $N$ does not have any positive eigenvalues, i.e., $N \leq 0$. We clearly have that $x^\top M x > 0$ for all nonzero $x \in \mathbb{R}^{q+r}$ satisfying $x^\top N x = 0$. Then, by Lemma~\ref{l:standardFinsler'slemma}, there exists a real $\bar{\alpha} \in \mathbb{R}$ such that $M-\bar{\alpha} N > 0$. If $\bar{\alpha} \geq 0$ then we have $M-\alpha N > 0$ for $\alpha = \bar{\alpha}$. On the other hand, if $\bar{\alpha} < 0$ then $M > \bar{\alpha} N \geq 0$, so $M - \alpha N > 0$ for $\alpha = 0$. This proves the theorem.  
\end{proof}

One can even prove a strict matrix S-lemma in the case that $N_{22}$ is not necessarily negative definite, but under the extra assumptions that $M_{22} \leq 0$ and the Slater condition holds on $N$. It turns out, however, that in that case we need two real numbers $\alpha \geq 0$ and $\beta > 0$ to state a necessary and sufficient condition.

\begin{theorem}[Strict matrix S-lemma with $\alpha$ and $\beta$]
\label{t:strictS-lemma}
Let $M,N \in  \mathbb{S}^{q+r}$. Then we have that $\calZ_r(N) \subseteq \calZ_r^+(M)$
if there exist scalars  $\alpha \geq 0$ and $\beta >0$ such that \vspace{-5pt}
\begin{equation}
\label{ineqalphabeta}
\vspace{-5pt}M-\alpha N \geq \begin{bmatrix}
\beta I & 0 \\ 0 & 0
\end{bmatrix}.
\end{equation} 
Assume, in addition, that $N \in \bpi_{q,r}$, $M_{22} \leq 0$ and $N$ has at least one positive eigenvalue. Then $\calZ_r(N) \subseteq \calZ_r^+(M)$ if and only if there exist $\alpha \geq 0$ and $\beta >0$ such that \eqref{ineqalphabeta} holds.
\end{theorem}

\begin{proof}
Both `if' statements are clear, so we focus on the `only if' part. Assume  that $\calZ_r(N) \subseteq \calZ_r^+(M)$. We will first prove that $\ker N_{22} \subseteq \ker M_{22}$ and $\ker N_{22} \subseteq \ker M_{12}$. Let $Z \in \calZ_r(N)$ and $v \in \ker N_{22}$. In addition, select any nonzero vector $w \in \mathbb{R}^q$ and define $\hat{Z} := vw^\top$. Since $N \in \bpi_{q,r}$, we have that $Z+\gamma \hat{Z} \in \calZ_r(N)$ for all $\gamma \in \mathbb{R}$. Therefore $Z+\gamma \hat{Z} \in \calZ_r^+(M)$. We write \vspace{-5pt}
\begin{equation}
\label{ineqZZhat}
\vspace{-5pt}0 < \begin{bmatrix}
I \\ Z + \gamma \hat{Z}
\end{bmatrix}^\top M \begin{bmatrix}
I \\ Z + \gamma \hat{Z}
\end{bmatrix} = \calL(\gamma) + \gamma^2 (v^\top M_{22} v) \: w w^\top,
\end{equation} 
where $\calL(\gamma)$ is a matrix that depends affinely on $\gamma$. This implies that $M_{22} v = 0$. Indeed, if $M_{22}v \neq 0$ then $v^\top M_{22} v < 0$ and we can find a sufficiently large $\gamma \in \mathbb{R}$ that violates \eqref{ineqZZhat}. We conclude that $\ker N_{22} \subseteq \ker M_{22}$. Next, let $v \in \ker N_{22}$ and define $\hat{Z} := - vv^\top M_{21}$. Since $v \in \ker M_{22}$, we can write\vspace{-5pt}
\begin{equation}
\label{ineq2gamma}
\vspace{-5pt}0 < \begin{bmatrix}
I \\ Z + \gamma \hat{Z}
\end{bmatrix}^\top M \begin{bmatrix}
I \\ Z + \gamma \hat{Z}
\end{bmatrix} =  \begin{bmatrix}
I \\ Z 
\end{bmatrix}^\top M \begin{bmatrix}
I \\ Z 
\end{bmatrix} - 2 \gamma \: M_{12} vv^\top M_{21}.
\end{equation}
This implies that $M_{12}v = 0$, for otherwise we can select a sufficiently large $\gamma \in \mathbb{R}$ violating \eqref{ineq2gamma}. Therefore, we conclude that $\ker N_{22} \subseteq \ker M_{12}$. 
Subsequently, we claim that there exists a scalar $\beta > 0$ such that \vspace{-5pt}
\begin{equation}
\label{claimbeta}
\vspace{-5pt}\calZ_r(N) \subseteq \calZ_r^+\left(M-\begin{bmatrix}
\beta I & 0 \\ 0 & 0
\end{bmatrix}\right).
\end{equation} 
Suppose on the contrary that this claim is false. Then there exists a sequence $\{\beta_i\}$ such that $\beta_i \to 0$ ($i\to \infty$) and for all $i$ there exists $Z_i \in \calZ_r(N)$ such that \vspace{-5pt}
\begin{equation}
\label{ZinotinZr}
\vspace{-5pt}Z_i \not\in \calZ_r^+\left(M-\begin{bmatrix}
\beta_i I & 0 \\ 0 & 0
\end{bmatrix}\right).
\end{equation}
Define $\calV := \{Z \in \mathbb{R}^{r \times q} \mid N_{22} Z = 0\}$. Write $Z_i$ as $Z_i = Z_i^1 + Z_i^2$ where $Z_i^1 \in \calV^\perp$ and $Z_i^2 \in \calV$. Since $\ker N_{22} \subseteq \ker N_{12}$ we see that $Z_i^1 \in \calZ_r(N)$ for all $i$. Next, we claim that $\{Z_i^1\}$ is bounded. We will prove this by contradiction. Thus, assume that $\{Z_i^1\}$ is unbounded. Clearly, the sequence $\left\{\frac{Z_i^1}{\norm{Z_i^1}}\right\}$ is bounded. By Bolzano-Weierstrass, it thus has a convergent subsequence with limit, say $Z_*$. Note that\vspace{-5pt}
\[\vspace{-5pt}\frac{1}{\norm{Z_i^1}^2}(N_{11} + N_{12} Z_i^1 + (N_{12} Z_i^1)^\top + (Z_i^1)^\top N_{22} Z_i^1) \geq 0. 
\]
By taking the limit along the subsequence as $i \to \infty$, we obtain $Z_*^\top N_{22}Z_* \geq 0$. Using the fact that $N_{22} \leq 0$, we conclude that $Z_* \in \calV$. Since $Z_i^1 \in \calV^\perp$ for all $i$, also $\frac{Z_i^1}{\norm{Z_i^1}} \in \calV^\perp$ and thus $Z_* \in \calV^\perp$. Therefore, we conclude that both $Z_* \in \calV$ and $Z_* \in \calV^\perp$. That is, $Z_* = 0$. This is a contradiction as $\frac{Z_i^1}{\norm{Z_i^1}}$ has norm $1$ for all $i$. We conclude that the sequence $\{Z_i^1\}$ is bounded. It thus contains a convergent subsequence with limit, say $Z_*$. Note that $\calZ_r(N)$ is closed and thus $Z_* \in \calZ_r(N)$. Since $\ker N_{22} \subseteq \ker M_{22}$ and $\ker N_{22} \subseteq \ker M_{12}$, \eqref{ZinotinZr} implies that \vspace{-5pt}
\[\vspace{-5pt}
\begin{bmatrix} I \\
Z_i^1 \end{bmatrix}^\top \left(M-\begin{bmatrix}
\beta_i I & 0 \\ 0 & 0 
\end{bmatrix}
\right)\begin{bmatrix} I \\
Z_i^1 \end{bmatrix} \not> 0
\]
for all $i$. We take the limit as $i \to \infty$ along a subsequence with limit $Z_*$, which yields $Z_* \not\in \calZ_r^+(M)$. However, since $Z_* \in \calZ_r(N)$, this contradicts our hypothesis that $\calZ_r(N) \subseteq \calZ_r^+(M)$. Therefore, we conclude that there exists a $\beta > 0$ such that \eqref{claimbeta} holds. In particular, this implies that there exists $\beta > 0$ such that\vspace{-5pt}
\[\vspace{-5pt}\calZ_r(N) \subseteq \calZ_r\left(M-\begin{bmatrix}
\beta I & 0 \\ 0 & 0
\end{bmatrix}\right).\]
Finally, by Theorem~\ref{t:nonstrictS-lemma}, there exists a scalar $\alpha \geq 0$ such that \eqref{ineqalphabeta} holds. 
\end{proof}

Next, we state a matrix Finsler's lemma in the case of a strict inequality. 

\begin{theorem}[Strict matrix Finsler's lemma]
\label{t:strictmatFinsler}
Let $M,N \in  \mathbb{S}^{q+r}$. Then $\calZ_r^0(N) \subseteq \calZ_r^+(M)$
if there exist scalars  $\alpha \in \mathbb{R}$ and $\beta >0$ such that \eqref{ineqalphabeta} holds. Next, assume that $N \in \bpi_{q,r}$, $N\schur N_{22} = 0$ and $M_{22} \leq 0$. Then $\calZ_r^0(N) \subseteq \calZ_r^+(M)$ if and only if there exist $\alpha \geq 0$ and $\beta >0$ such that \eqref{ineqalphabeta} holds.
\end{theorem}

\begin{proof} 
The `if' statements are obvious. To prove the `only if' statement, assume that $\calZ_r^0(N) \subseteq \calZ_r^+(M)$. Since $N \schur N_{22} = 0$ and $N \in \bpi_{q,r}$, we have that $N \leq 0$. This implies that $\calZ_r^0(N) = \calZ_r(N)$. Therefore, we also have that $\calZ_r(N) \subseteq \calZ_r^+(M)$. We can thus use the same argument as in the proof of Theorem~\ref{t:nonstrictS-lemma} to show that $\ker N_{22} \subseteq \ker M_{22}$ and $\ker N_{22} \subseteq \ker M_{12}$. Let $T\!\! := \!\begin{bmatrix}I & 0 \\ -N_{22}^\dagger N_{21}\!\!\!\! & I \end{bmatrix}$ and $\Theta$ as~\eqref{eq:def Theta}, then
$$
T^\top N T = \begin{bmatrix}
0 & 0 \\ 0 & N_{22}
\end{bmatrix} \text{ and } T^\top M T = \begin{bmatrix}
\Theta & M_{12} -N_{12} N_{22}^\dagger M_{22} \\
M_{21}-M_{22} N_{22}^\dagger N_{21} & M_{22}
\end{bmatrix}.
$$
Since $-N_{22}^\dagger N_{21} \in \calZ_r^0(N) \subseteq \calZ_r^+(M)$ we have $\Theta > 0$. Then obviously, there exists a real $\beta > 0$ so that $\Theta - \beta I > 0$. We have that
\begin{equation*}
T^\top \left(M-\alpha N - \begin{bmatrix}
\beta I & 0 \\ 0 & 0
\end{bmatrix} \right) T = \begin{bmatrix}
\Theta - \beta I & M_{12} -N_{12} N_{22}^\dagger M_{22} \\
M_{21}-M_{22} N_{22}^\dagger N_{21} & M_{22} - \alpha N_{22}
\end{bmatrix}.
\end{equation*}
Therefore it holds that $T^\top \left( M-\alpha N - \begin{bmatrix} \beta I & 0 \\ 0 & 0 \end{bmatrix} \right) T \geq 0$ if and only if 
\begin{equation}
\label{strictineqM22N22}
M_{22} - \alpha N_{22} - (M_{21}-M_{22} N_{22}^\dagger N_{21}) \: (\Theta-\beta I)^{-1} (M_{12} -N_{12} N_{22}^\dagger M_{22}) \geq 0. 
\end{equation}

Because $N_{22} \leq 0$, $\ker N_{22} \subseteq \ker M_{22}$ and $\ker N_{22} \subseteq \ker M_{12}$, there exists a sufficiently large $\alpha \geq 0$ such that \eqref{strictineqM22N22} holds. This proves the statement.
\end{proof}

Finally, we note that it is possible to combine the strict versions of the matrix S-lemma and Finsler's lemma, Theorems~\ref{t:strictS-lemma} and \ref{t:strictmatFinsler}, into one result. This results in the following corollary. Note the absence of the Slater condition on $N$. 

\begin{corollary}
\label{c:combinedstrictS-lemmaFinslerslemma}
Let $M,N \in  \mathbb{S}^{q+r}$. Then $\calZ_r(N) \subseteq \calZ_r^+(M)$
if there exist scalars  $\alpha \geq 0$ and $\beta >0$ such that \eqref{ineqalphabeta} holds. Next, assume that $N \in \bpi_{q,r}$ and $M_{22} \leq 0$. Then $\calZ_r(N) \subseteq \calZ_r^+(M)$ if and only if there exist $\alpha \geq 0$ and $\beta >0$ such that \eqref{ineqalphabeta} holds.
\end{corollary}

\begin{proof}
Once again, the `if' parts are clear. To prove the `only if' statement, we distinguish the cases that $N$ has at least one positive eigenvalue, and $N \leq 0$ (equivalently, $N \schur N_{22} = 0$). In the first case, Theorem~\ref{t:strictS-lemma} is directly applicable, resulting in the existence of $\alpha \geq 0$ and $\beta >0$ such that \eqref{ineqalphabeta} holds. In the second case, $\calZ_r(N) = \calZ_r^0(N)$ and Theorem~\ref{t:strictmatFinsler} yields $\alpha \geq 0$ and $\beta >0$ satisfying \eqref{ineqalphabeta}. 
\end{proof}

\subsection{Comparison with related work}
\label{s:comparison}

\subsubsection{Full block S-procedure and LMI relaxations}

In this section, we point out a relation between Theorem~\ref{t:strictS-lemmaN22} and the literature on LMI relaxations in robust control, see \cite{Scherer2005,Scherer2006,Scherer2006b}. First of all, we acknowledge that these references are able to deal with more general uncertainty sets than those defined by quadratic matrix inequalities (i.e., the set $\calZ_r(N)$). Nevertheless, we can apply the general theory of \cite{Scherer2005} to obtain a matrix S-lemma with a strict inequality on $M$. To obtain the result, we substitute $A = 0$, $B = I$, $W(x) = -M$ and \vspace{-5pt}
\[\vspace{-5pt}{\bf\Delta} = \left\{\Delta \mid \begin{bmatrix}
	\Delta \\ I 
	\end{bmatrix}^\top N \begin{bmatrix}
	\Delta \\ I 
	\end{bmatrix}\geq 0\right\}\]
into Equation (1.2) of \cite{Scherer2005}. Then, we combine the full block S-procedure (c.f. \cite{Scherer2001},\cite[p. 367]{Scherer2005}) with the fact that the LMI relaxation of \cite{Scherer2005} is exact for a single full block \cite[Thm. 5.3]{Scherer2005}. This yields the following result. 

\begin{proposition}
\label{p:Scherer}
Let $M,N \in \mathbb{S}^{q+r}$. Assume that $N$ is nonsingular, $N_{11} \geq 0$ and $N_{22} < 0$. Then we have that $\calZ_r(N) \subseteq \calZ_r^+(M)$ if and only if there exists a real $\alpha \geq 0$ such that $M - \alpha N > 0$.
\end{proposition}

We note that Proposition~\ref{p:Scherer} also follows from Theorem~\ref{t:strictS-lemmaN22} of this paper. In fact, $N_{11} \geq 0$ and $N_{22} < 0$ imply that $N \in \bpi_{q,r}$, making Theorem~\ref{t:strictS-lemmaN22} readily applicable. Theorem~\ref{t:strictS-lemmaN22}, however, is more general in the sense that nonsingularity of $N$ and $N_{11} \geq 0$ are not required in general as long as $N \schur N_{22} \geq 0$.

\subsubsection{Other matrix S-lemmas and Finsler's lemmas}

In the paper \cite{vanWaarde2022}, various matrix S-lemmas were presented. The main contribution of Section~\ref{s:FinslerandSlemma} of this paper is to weaken some of the assumptions of these results, especially the so-called \emph{generalized Slater condition}. We illustrate this by recalling \cite[Thm. 13]{vanWaarde2022}.

\begin{proposition}
\label{p:S-lemmapaper}
Let $M,N \in \mathbb{S}^{q+r}$. Assume that $\ker N_{22} \subseteq \ker N_{12}$, $N_{22} \leq 0$ and $M_{22} \leq 0$. Suppose that $\calZ_r^+(N)$ is nonempty. Then we have that $\calZ_r(N) \subseteq \calZ_r^+(M)$ if and only if there exist $\alpha \geq 0$ and $\beta > 0$ such that \eqref{ineqalphabeta} holds.
\end{proposition}

In \cite{vanWaarde2022}, the nonemptiness of $\calZ_r^+(N)$ was referred to as the generalized Slater condition. This condition, combined with the assumptions that $N_{22} \leq 0$ and $\ker N_{22} \subseteq \ker N_{12}$, implies that $N \schur N_{22} > 0$ (see \eqref{e:expression for ztpiz}). As such, the conditions of Proposition~\ref{p:S-lemmapaper} imply that $N \in \bpi_{q,r}$. Proposition~\ref{p:S-lemmapaper} thus follows immediately from the more general Corollary~\ref{c:combinedstrictS-lemmaFinslerslemma}. The relevance of Corollary~\ref{c:combinedstrictS-lemmaFinslerslemma} is that the condition $N \schur N_{22} > 0$ has been replaced by the less restrictive assumption $N \schur N_{22} \geq 0$. In Section~\ref{sec:appl}, we will illustrate the benefits of this weaker condition.

In addition to \cite[Thm. 13]{vanWaarde2022}, also a matrix S-lemma for non-strict inequalities was presented in \cite[Thm. 9]{vanWaarde2022}. This matrix S-lemma is not directly comparable to Theorem~\ref{t:nonstrictS-lemma} of this paper because it operates under different assumptions. In fact, \cite[Thm. 9]{vanWaarde2022} does not assume that $N \in \bpi_{q,r}$, but does instead assume the generalized Slater condition. For the application of data-driven control that we have in mind, however, the condition $N \in \bpi_{q,r}$ is always satisfied (see Section~\ref{sec:appl}). In this case, the relevance of Theorem~\ref{t:nonstrictS-lemma} lies in the fact that it only assumes the ``ordinary" Slater condition ($N$ has at least one positive eigenvalue), rather than the generalized Slater condition which requires $N$ to have at least $q$ positive eigenvalues.

Finally, in \cite{vanWaarde2021}, a matrix version of Finsler's lemma was provided. This results, however, operates under rather strong conditions such as $M_{11} > 0$ and $M_{12} = 0$. It can be shown that \cite[Thm. 1]{vanWaarde2021}, in fact, follows from Theorem~\ref{t:matFinsler} of this paper.

\subsubsection{Petersen's lemma}

In the recent paper \cite{Bisoffi2021}, the problem of data-driven stabilization is tackled using a result known as \emph{Petersen's lemma}, which originates from \cite{Petersen1987,Petersen1986}. In \cite{Bisoffi2021}, a non-strict and a strict version of Petersen's lemma are stated. These results are relevant for this paper, because they are intimately related to matrix S-lemmas. In fact, as we will see shortly, they can be regarded as matrix S-lemmas for special choices of the matrices $M$ and $N$, given in \eqref{petersenMN}. 

Petersen's lemma, as well as the resulting conditions for data-driven stabilization \cite{Bisoffi2021}, are interesting because they work under different assumptions than the matrix S-lemmas of \cite{vanWaarde2022}. Specifically, the results of \cite{vanWaarde2022} apply to unbounded sets $\calZ_r(N)$ but assume the generalized Slater condition. On the other hand, when considering the results of \cite{Bisoffi2021} in the framework of this paper, the particular noise model \cite[Eq. 8]{Bisoffi2021} and assumption on the data \cite[As. 1]{Bisoffi2021} imply that the resulting set $\calZ_r(N)$ is bounded. However, Petersen's lemma does not require the generalized Slater condition (only the ``ordinary" Slater condition is required in the non-strict Petersen's lemma, while the strict version does not require a Slater condition). 

The results in Section~\ref{s:FinslerandSlemma} of the current paper can be regarded as the best of both worlds. In fact, we will show in Proposition~\ref{p:Petersenslemma} that both the non-strict and the strict Petersen's lemma can be obtained directly from Theorems~\ref{t:nonstrictS-lemma} and \ref{t:strictS-lemmaN22}, respectively. With respect to the literature on Petersen's lemma, we conclude that the results of Section~\ref{s:FinslerandSlemma} are relevant because i) they are more general (i.e., they recover Petersen's lemmas as special cases, but also apply to unbounded sets $\calZ_r(N)$), and ii) they appear \emph{naturally} in data-driven control applications that can often be formulated as implications involving QMI's. In comparison, the application of Petersen's lemma requires a less straightforward reformulation of the set \eqref{ineqAB}, cf.
\cite{Bisoffi2021}.

\begin{proposition}[Petersen's lemma]
\label{p:Petersenslemma}
Consider $C \in \mathbb{R}^{n \times n}$, $E \in \mathbb{R}^{n \times p}$, $\bar{F} \in \mathbb{R}^{q \times q}$ and $G \in \mathbb{R}^{q \times n}$ with $C = C^\top$ and $\bar{F} \geq 0$, and define $\calF := \{F\in \mathbb{R}^{p \times q} \mid F^\top F \leq \bar{F} \}$.
\begin{enumerate}[label=\emph{(\alph*)}, ref=\alph*]
\item\label{p:Petersenslemma.1} The two following statements are equivalent\vspace{-5pt}
\begin{subequations}
\begin{align}\label{strictpetersenineq}
C+EFG+G^\top F^\top E^\top < 0 \:\:\text{ for all } F \in \calF,\\
\label{lambdastrictpetersen}
\textrm{there exists } \lambda>0 \textrm{ s.t. } C+\lambda EE^\top + \lambda^{-1} G^\top \bar{F}G < 0.
\end{align}
\end{subequations}
\item\label{p:Petersenslemma.2} \vspace{-5pt} Suppose that $E \neq 0$, $\bar{F} > 0$ and $G \neq 0$. Then the following are equivalent\vspace{-5pt}
\begin{subequations}
\begin{align}
\label{nonstrictpetersenineq}
C+EFG+G^\top F^\top E^\top \leq 0 \:\: \text{ for all } F \in \calF, \\
\label{lambdanonstrictpetersen}
\textrm{there exists } \lambda>0 \textrm{ s.t. } C+\lambda EE^\top + \lambda^{-1} G^\top \bar{F}G \leq 0.
\end{align}\end{subequations}\end{enumerate}\end{proposition}

\vspace{-5pt}\begin{proof}
First, note that by applying Lemma~\ref{l:A A <= B B 2} to $A = F$ and $B = \bar{F}^\half$, we get \vspace{-5pt}
\[\vspace{-5pt}\{FG \in \mathbb{R}^{p \times n} \!\mid\! F \in \mathbb{R}^{p \times q} \text{ and } F^\top F \leq \bar{F}\} = \{ S\bar{F}^\half G \in \mathbb{R}^{p \times n} \!\mid\! S \in \mathbb{R}^{p \times q} \text{ and } S^\top S \leq I \}.\]
Applying the same lemma with $A = H$ and $B = \bar{F}^\half G$, this set is equal to $\{H \in \mathbb{R}^{p \times n} \mid H^\top H \leq G^\top \bar{F} G \}$. Therefore, by defining the matrices in \eqref{partitionMN} as\vspace{-5pt}
\begin{equation}
\label{petersenMN}
\vspace{-5pt}M := \begin{bmatrix}
-C & -E \\ -E^\top & 0
\end{bmatrix} \text{ and } N := \begin{bmatrix}
G^\top \bar{F} G & 0 \\ 0 & -I
\end{bmatrix},
\end{equation}
we see that \eqref{strictpetersenineq} is equivalent to the inclusion $\calZ_p(N) \subseteq \calZ_p^+(M)$, while \eqref{nonstrictpetersenineq} is equivalent to $\calZ_p(N) \subseteq \calZ_p(M)$. Note that $N \in \bpi_{n,p}$ and $N_{22} < 0$. 

Now, we prove statement \eqref{p:Petersenslemma.1}. By Theorem~\ref{t:strictS-lemmaN22}, $\calZ_p(N) \subseteq \calZ_p^+(M)$ is equivalent to the existence of $\alpha \geq 0$ such that $M-\alpha N > 0$. In fact, $\alpha$ is necessarily positive because $M_{22} = 0$ and $N_{22} < 0$. By using a Schur complement argument, $M-\alpha N > 0$ is equivalent to $-C-\alpha G^\top \bar{F} G -\frac{1}{\alpha} EE^\top > 0$, thus \eqref{lambdastrictpetersen} holds for $\lambda := \frac{1}{\alpha}$.

The proof of statement \eqref{p:Petersenslemma.2} is analogous, but uses Theorem~\ref{t:nonstrictS-lemma} instead of Theorem~\ref{t:strictS-lemmaN22}. Note that the conditions of that theorem are satisfied since $\bar{F} > 0$ and $G \neq 0$ imply that $N$ has at least one positive eigenvalue. Moreover, $M - \alpha N \geq 0$, $E \neq 0$, and $\alpha \geq 0$ imply that $\alpha > 0$. This proves the proposition. 
\end{proof}

\section{Applications to data-driven control}\label{sec:appl}

In this section we apply the results to data-driven control problems. For brevity, we restrict our attention to data-based stabilization. However, we note that the results of this paper are also relevant for other analysis and design problems such as data-based dissipativity analysis \cite{Koch2021,vanWaarde2022b}, $H_2$- and $H_\infty$-control \cite{vanWaarde2022}, and even data-driven reduced order modeling \cite{Burohman2021}. In Section~\ref{sec:stab} we provide a general necessary and sufficient condition for informativity for quadratic stabilization. Subsequently, in Section~\ref{sec:reduction} we provide alternative conditions for quadratic stabilization that are more appealing from a computational point of view. We also provide an explicit formula for the controller. 

\subsection{Data-driven quadratic stabilization}\label{sec:stab}

Using the theory of Section~\ref{s:FinslerandSlemma} we want to understand i) under which conditions $(U_-,X)$ are informative for quadratic stabilization, and ii) how to construct a controller $K$ given informative data. Throughout this section, we will assume that $\Phi \in \bpi_{n,T}$. 
Define the matrices\vspace{-5pt}
\begin{align}
    M :=& 
\begin{pmat}[{|}]
M_{11} & M_{12} \cr\-
M_{21} & M_{22} \cr
\end{pmat} := \small \begin{pmat}[{|.}]
    P & 0 & 0 \cr\-
    0 & -P & -PK^\top \cr
    0 & -KP & -KPK^\top \cr
    \end{pmat}\normalsize\label{defM}  \\ 
N :=& \begin{pmat}[{|}]
N_{11} & N_{12} \cr\- N_{21} & N_{22} \cr
\end{pmat} 
:= \small \begin{pmat}[{.}]
    I & X_+ \cr\- 0 & -X_- \cr 0 & -U_- \cr
    \end{pmat}
    \begin{bmatrix}
    \Phi_{11} & \Phi_{12} \\
    \Phi_{21} & \Phi_{22}
    \end{bmatrix}
    \begin{pmat}[{.}]
    I & X_+ \cr\- 0 & -X_- \cr 0 & -U_- \cr
    \end{pmat}^\top.\normalsize \label{defN}
\end{align}
Recall from Section~\ref{s:motivation} and \eqref{ineqAB} that $\Sigma=\calZ_{n+m}(N)$. Next, suppose that we fix a Lyapunov matrix $P > 0$ and a feedback gain $K$. The inequality \eqref{lyapunovineq} is equivalent to $\begin{bmatrix} A & B \end{bmatrix}^\top \in \calZ_{n +m}^+(M)$.
Then we need to find conditions on the data such that there exist $P >0$ and $K$ for which the inclusion \vspace{-5pt}
\begin{equation} \label{QMI inclusion}
\vspace{-5pt}\calZ_{n +m}(N) \subseteq \calZ_{n +m}^+(M)
\end{equation}
holds. In order to find such conditions, we will apply Corollary~\ref{c:combinedstrictS-lemmaFinslerslemma}. 
To do so, we need to verify its assumptions. In particular, we need to verify that  $N_{22} \leq 0$, $N\schur N_{22} \geq 0$, $\ker N_{22} \subseteq \ker N_{12}$ and  $M_{22} \leq 0$. Note that 
$N_{22} = \begin{bmatrix}X_- \\ U_-\end{bmatrix} \Phi_{22} \begin{bmatrix}X_- \\ U_-\end{bmatrix}^\top \leq 0$ because $\Phi_{22} \leq 0$. This implies that\vspace{-5pt}
\[\vspace{-5pt}\ker N_{22} = \ker \Phi_{22}\begin{bmatrix}
X_- \\ U_-
\end{bmatrix}^\top, \text{ and } \ker N_{12} = \ker \bigg( (\Phi_{12}+X_+ \Phi_{22})\begin{bmatrix}
X_- \\ U_-
\end{bmatrix}^\top \bigg).\]
Since $\ker \Phi_{22} \subseteq \ker \Phi_{12}$, also $\ker N_{22} \subseteq \ker N_{12}$. In addition, \eqref{e:obs gen schur} and $(A_s,B_s) \in \Sigma$ imply $N\schur N_{22}\geq 0$. Since $P >0$, $M_{22} = - \begin{bmatrix}
I \\ K
\end{bmatrix} P \begin{bmatrix}
I \\ K
\end{bmatrix}^\top \leq 0$. Now, Corollary \ref{c:combinedstrictS-lemmaFinslerslemma} asserts that \eqref{QMI inclusion} holds if and only if there exist scalars $\alpha \geq 0$ and $\beta > 0$ such that \vspace{-5pt}
\begin{equation}
\label{ineqMNab}
\vspace{-5pt}M - \alpha N \geq \begin{bmatrix}
\beta I_n & 0 \\ 0  &  0_{(n+m)\times(n+m)}
\end{bmatrix}.
\end{equation}
From a design point of view, the matrices $P$ and $K$ that appear in $M$ are not given. However, the idea is now to \emph{compute} matrices $P$, $K$ and scalars $\alpha$ and $\beta$ such that \eqref{ineqMNab} holds. In fact, by the above discussion, the data $(U_-,X)$ are informative for quadratic stabilization \emph{if and only if} there exists an $n \times n$ matrix $P > 0$, a $K \in \mathbb{R}^{m \times n}$ and two scalars $\alpha \geq 0$ and $\beta > 0$ such that \eqref{ineqMNab} holds. We note that \eqref{ineqMNab} (in particular, $M$) is not linear in $P$ and $K$. Nonetheless, by a rather standard change of variables and a Schur complement argument, we can transform \eqref{ineqMNab} into a linear matrix inequality. 
Moreover, it turns out that the scalar $\alpha$ is necessarily positive. By a scaling argument then, it can be chosen to be equal to 1.  We summarize this in the following theorem.

\begin{theorem}
\label{t:theoremstab}
Let $\Phi \in \bpi_{n,T}$ and consider the data $(U_-,X)$, generated by \eqref{systemprocessnoise} with noise model \eqref{asnoise}. The following statements hold. 
\begin{enumerate}[label=\emph{(\alph*)}, ref=\alph*,leftmargin=*]
\item \label{i:infstab1} The data $(U_-,X)$ are informative for quadratic stabilization if and only if there exists an $n\times n$ matrix $P > 0$, an $L \in \mathbb{R}^{m \times n}$ and a scalar 
$\beta > 0$ satisfying \vspace{-5pt}
\begin{equation}\label{LMIstab}
\vspace{-5pt} \small\begin{bmatrix}
    P-\beta I & 0 & 0 & 0 \\
    0 & -P & -L^\top & 0 \\
    0 & -L & 0 & L \\
    0 & 0 & L^\top & P
    \end{bmatrix} \! - \! \begin{bmatrix}
    I & X_+ \\ 0 & -X_- \\ 0 & -U_- \\ 0 & 0
    \end{bmatrix}\!
    \begin{bmatrix}
    \Phi_{11} & \Phi_{12} \\
    \Phi_{21} & \Phi_{22}
    \end{bmatrix}\!
    \begin{bmatrix}
    I & X_+ \\ 0 & -X_- \\ 0 & -U_- \\ 0 & 0
    \end{bmatrix}^\top \geq 0 \normalsize . 
\end{equation}
\item \label{i:infstab2} Assume, in addition, that $\Phi_{22} < 0$ and
$
\rank \begin{bmatrix}
X_-^\top & U_-^\top
\end{bmatrix}^\top = n+m
$.
Then the data $(U_-,X)$ are informative for quadratic stabilization if and only if there exists an $n\times n$ matrix $P > 0$ and a matrix $L \in \mathbb{R}^{m \times n}$ satisfying \vspace{-5pt}
\begin{equation}\label{LMIstabstrict}
\vspace{-5pt}\small \begin{bmatrix}
    P & 0 & 0 & 0 \\
    0 & -P & -L^\top & 0 \\
    0 & -L & 0 & L \\
    0 & 0 & L^\top & P
    \end{bmatrix} - \begin{bmatrix}
    I & X_+ \\ 0 & -X_- \\ 0 & -U_- \\ 0 & 0
    \end{bmatrix}
    \begin{bmatrix}
    \Phi_{11} & \Phi_{12} \\
    \Phi_{21} & \Phi_{22}
    \end{bmatrix}
    \begin{bmatrix}
    I & X_+ \\ 0 & -X_- \\ 0 & -U_- \\ 0 & 0
    \end{bmatrix}^\top > 0 \normalsize . 
\end{equation}
\end{enumerate}
Moreover, if either \eqref{LMIstab} or \eqref{LMIstabstrict} is feasible then $K := L P^{-1}$ is a stabilizing feedback gain for all $(A,B) \in \Sigma$.
\end{theorem}

\begin{proof}
We first focus on statement \eqref{i:infstab1}. To prove the `if' part, suppose that there exist $P$, $L$ and $\beta$ satisfying \eqref{LMIstab}. Define $K := L P^{-1}$. By computing the Schur complement of \eqref{LMIstab} with respect to its fourth diagonal block, we obtain \eqref{ineqMNab} with $\alpha = 1$. As such, \eqref{QMI inclusion} holds. We conclude that the data $(U_-,X)$ are informative for quadratic stabilization and $K = LP^{-1}$ is indeed a stabilizing controller for all $(A,B) \in \Sigma$. 

Conversely, to prove the `only if' part, suppose that the data $(U_-,X)$ are informative for quadratic stabilization. This means that there exist $P > 0$ and $K$ such that \eqref{QMI inclusion} holds. By Corollary~\ref{c:combinedstrictS-lemmaFinslerslemma} there exist $\alpha \geq 0$ and $\beta >0$ satisfying \eqref{ineqMNab}. Then, by defining $L := KP$ and using a Schur complement argument, we conclude that\vspace{-5pt}
\begin{equation}
\label{conditionstabwithalpha}
\vspace{-5pt}\small \begin{bmatrix}
    P-\beta I & 0 & 0 & 0 \\
    0 & -P & -L^\top & 0 \\
    0 & -L & 0 & L \\
    0 & 0 & L^\top & P
    \end{bmatrix} - \alpha \begin{bmatrix}
    I & X_+ \\ 0 & -X_- \\ 0 & -U_- \\ 0 & 0
    \end{bmatrix}
    \begin{bmatrix}
    \Phi_{11} & \Phi_{12} \\
    \Phi_{21} & \Phi_{22}
    \end{bmatrix}
    \begin{bmatrix}
    I & X_+ \\ 0 & -X_- \\ 0 & -U_- \\ 0 & 0
    \end{bmatrix}^\top \geq 0 \normalsize
\end{equation} 
is feasible. 
Zooming in on the $(2,2)$ block, this yields $-P - \alpha X_- \Phi_{22} X_-^\top \geq 0$. 
Since, by assumption, $\Phi_{22} \leq 0$, the inequality $P > 0$ implies $\alpha >0$. As a consequence, by scaling $P$, $L$ and $\beta$ by $\frac{1}{\alpha}$
we may assume that $\alpha =1$, so the LMI \eqref{LMIstab} is feasible.

The proof of statement \eqref{i:infstab2} follows the same steps as the proof of \eqref{i:infstab1}, with the difference that Theorem~\ref{t:strictS-lemmaN22} is invoked instead of Corollary~\ref{c:combinedstrictS-lemmaFinslerslemma}, since the assumptions of \eqref{i:infstab2} imply that $\begin{bmatrix}
X_- \\ U_- 
\end{bmatrix} \Phi_{22} \begin{bmatrix}
X_- \\ U_- 
\end{bmatrix}^\top < 0$. This proves the theorem. 
\end{proof}

Theorem~\ref{t:theoremstab} provides linear matrix inequality conditions for informativity for quadratic stabilization. Such feasibility problems have been studied extensively (see e.g. \cite{Boyd1994}) given their importance in the optimization literature. To be precise, LMI's arise as the constraints in second-order cone programming, or more generally, in semi-definite programming \cite{Boyd1994}. Because of this importance, powerful computational tools such as the solvers Sedumi \cite{Sturm1999} and Mosek have been developed in order to verify LMI feasibility. Theorem~\ref{t:theoremstab}\eqref{i:infstab1} provides a genuine necessary and sufficient condition for quadratic stabilization, in the sense that no additional rank conditions on the data or assumptions on the noise model are needed. From a numerical point of view, however, the strict inequality of Theorem~\ref{t:theoremstab}\eqref{i:infstab2} may be preferred, since LMI solvers are known to be unreliable for LMIs that define feasible sets without interior points. Theorem~\ref{t:theoremstab} is an significant improvement of \cite[Thm. 14]{vanWaarde2022}. In fact, the condition \eqref{conditionstabwithalpha} is exactly the same as in  \cite{vanWaarde2022}, but Theorem~\ref{t:theoremstab} works under weaker assumptions. Indeed, in \cite{vanWaarde2022}, the generalized Slater condition was imposed, which is equivalent to $N$ having $n$ positive eigenvalues. Since $N_{22} \leq 0$, the generalized Slater condition implies that $N\schur N_{22} > 0$. In particular, this implies that the conditions of Corollary~\ref{c:combinedstrictS-lemmaFinslerslemma} are satisfied. It should be clear, however, that $N\schur N_{22} > 0$ is not required in Theorem~\ref{t:theoremstab}. 

The weaker assumptions of Theorem~\ref{t:theoremstab} are especially relevant in the case that $\bmw$ takes its values in an a priori given \emph{subspace} of the state-space, which is the case for noise model~\ref{i:nm4} in Section~\ref{s:motivation}. Indeed, suppose that $w(t) = E \hat{w}(t)$ for all $t = 0,1,2\dots,T-1$, where $\hat{w}(t) \in \mathbb{R}^d$ and $E \in \mathbb{R}^{n \times d}$ is a given matrix of full column rank. The matrix $
\hat{W}_- = \begin{bmatrix}
\hat{w}(0) & \hat{w}(1) & \cdots & \hat{w}(T-1)
\end{bmatrix}$ captures the noise. As before, $\hat{W}_-$ is unknown but is assumed to satisfy $\hat{W}_-^\top \in \calZ_T(\hat{\Phi})$, where $\hat{\Phi} \in \bpi_{d,T}$ is such that $\hat{\Phi}_{22} < 0$. Now, by Theorem~\ref{t:projectionfcr}, $W_- = E\hat{W}_-$ for some $\hat{W}_-^\top \in \calZ_T(\hat{\Phi})$ if and only if $W_-^\top \in \calZ_T(\Phi)$, where\vspace{-5pt}
\begin{equation}\label{PhiwithE}
\vspace{-5pt}\Phi := \begin{bmatrix}
E \hat{\Phi}_{11} E^\top & E \hat{\Phi}_{12} \\
\hat{\Phi}_{21} E^\top & \hat{\Phi}_{22} 
\end{bmatrix} \in \bpi_{n,T}.
\end{equation}
The conclusion is that Theorem~\ref{t:theoremstab} is applicable in cases where the noise is constrained to a known subspace, which is captured by the noise bound \eqref{asnoise} with $\Phi$ in \eqref{PhiwithE}. We note that \cite[Thm. 14]{vanWaarde2022} is generally not applicable in this context. The reason is that the matrix
$\Phi \schur \hat{\Phi}_{22} = E(\hat{\Phi}_{11} - \hat{\Phi}_{12}\hat{\Phi}_{22}^{-1}\hat{\Phi}_{21})E^\top$
can only be positive definite in the case that $d = n$. By \cite[Thm. 3.1]{Dancis1986}, this means that $N$ in \eqref{defN} always has less than $n$ positive eigenvalues whenever $d < n$.

\begin{remark}

We note that Theorem~\ref{t:theoremstab} makes use of state measurements $X$. It is possible to extend this result and to consider data-driven stabilization of input-output systems in autoregressive form $
P_s(\sigma)\bmy(t) = Q_s(\sigma) \bmu(t) + \bmw(t)$, where $\bmu \in \mathbb{R}^m$ is the input, $\bmy \in \mathbb{R}^p$ is the output and $\bmw \in \mathbb{R}^p$ is the noise. Here $P_s$ and $Q_s$ are unknown polynomial matrices and $\sigma$ denotes the shift operator, i.e., $(\sigma \bmf)(t) = \bmf(t+1)$. In this case, the data consist of samples of the inputs and outputs, and the goal is to use these data to construct a controller of the form $G(\sigma) \bmu(t) = F(\sigma) \bmy(t)$,
with $F$ and $G$ polynomial matrices, that renders the interconnected system stable in the sense that all solutions $(\bmu,\bmy)$ to \vspace{-5pt}
\[\vspace{-5pt}\begin{bmatrix}
G(\sigma) & -F(\sigma) \\ -Q_s(\sigma) & P_s(\sigma)
\end{bmatrix}\begin{bmatrix}
\bmu(t) \\ \bmy(t)
\end{bmatrix} = \begin{bmatrix}
0 \\ I
\end{bmatrix} \bmw(t)\]
converge to zero whenever $\bmw = 0$. This can be done by relying on tools for quadratic matrix inequalities, as developed in this paper, and the notion of quadratic difference forms \cite{Willems1998} from behavioral theory. We refer to \cite{vanWaarde2022d} for a detailed analysis of input-output systems. 

\end{remark}

\subsection{Reducing computational complexity}
\label{sec:reduction}
The computational complexity of determining feasibility of an LMI depends on the size of the LMI and the number of unknowns. The LMI \eqref{LMIstab}, together with the constraints $P> 0$ and $\beta > 0$, has size $4n+m+1$ and contains $\frac{n(n+1)}{2}+nm+1$ unknowns. Using Theorem~\ref{t:projectionfcr}, we can separate the computation of the Lyapunov matrix $P$ and the controller $K$. Below it will be shown that this leads to another LMI with size $4n+1$ and $\frac{n(n+1)}{2}+1$ unknowns. This result thus has a significant computational advantage over Theorem~\ref{t:theoremstab}.
\bthe
\label{t:reducedLMIs}
Let $\Phi \in \bpi_{n,T}$ and consider the data $(U_-,X)$, generated by \eqref{systemprocessnoise} with noise model \eqref{asnoise}. Define the matrix $\Theta := \Phi_{12}+X_+\Phi_{22}$. The following statements hold. 
\begin{enumerate}[label=\emph{(\alph*)}, ref=\alph*,leftmargin=*]
\item \label{i:reduction1} The data $(U_-,X)$ are informative for quadratic stabilization if and only if there exists an $n\times n$ matrix $P > 0$ and a scalar 
$\beta > 0$ satisfying \vspace{-10pt}
\begin{subequations}
	\begin{align} 
\small P\!-\!\beta I\!-\!
\bbm
I & X_+
\ebm
\!\Phi\!
\bbm
I \\ X^\top_+
\ebm
+
\Theta
\bbm
X_-\\U_-
\ebm^{\!\top}
\!\!\bigg(
\bbm
X_-\\U_-
\ebm
\Phi_{22}
\bbm
X_-\\U_-
\ebm^{\!\top}
\!\!\bigg)^{\!\!\scalebox{1.2}{$\gi$}}
\bbm
X_-\\U_-
\ebm
\Theta^\top\geq 0, \normalsize \label{e: reduced schur}\\
\small \begin{bmatrix}
    P-\beta I & 0 \\
    0 & -P 
    \end{bmatrix} - \begin{bmatrix}
    I & X_+ \\ 0 & -X_-  
    \end{bmatrix}
    \begin{bmatrix}
    \Phi_{11} & \Phi_{12} \\
    \Phi_{21} & \Phi_{22}
    \end{bmatrix}
    \begin{bmatrix}
    I & X_+ \\ 0 & -X_- 
    \end{bmatrix}^\top \geq 0.\normalsize\label{e: reduced proj} 
\end{align}
\end{subequations}
Moreover, if $P>0$ and $\beta>0$ satisfy \eqref{e: reduced schur} and \eqref{e: reduced proj} then \vspace{-5pt}
\begin{equation}
\label{constructionK}
\vspace{-5pt}K=\big(U_-(\Phi_{22}+\Theta^\top\Gamma\gi\Theta)X_-^\top\big)\big(X_-(\Phi_{22}+\Theta^\top\Gamma\gi\Theta)X_-^\top\big)\gi
\end{equation}
is a stabilizing gain for all $(A,B) \in \Sigma$, where $\Gamma=P-\beta I -
\bbm
I & X_+
\ebm
\Phi
\bbm
I \\ X^\top_+
\ebm$.
\item \label{i:reduction2} Assume, in addition, that $\Phi_{22} < 0$ and $
\rank \begin{bmatrix}
X_-^\top & U_-^\top
\end{bmatrix}^\top = n+m$. Then the data $(U_-,X)$ are informative for quadratic stabilization if and only if there exists an $n\times n$ matrix $P > 0$ satisfying\vspace{-5pt}
\begin{subequations}
\begin{align}
\small P -
\bbm
I & X_+
\ebm
\Phi
\bbm
I \\ X^\top_+
\ebm
+
\Theta
\bbm
X_-\\U_-
\ebm^{\!\top}
\!\!\bigg(
\bbm
X_-\\U_-
\ebm
\Phi_{22}
\bbm
X_-\\U_-
\ebm^{\!\top}
\!\!\bigg)^{-1}
\bbm
X_-\\U_-
\ebm
\Theta^\top > 0,\normalsize \label{e: reduced schur strict}\\
\small\begin{bmatrix}
    P & 0 \\
    0 & -P 
    \end{bmatrix} - \begin{bmatrix}
    I & X_+ \\ 0 & -X_-  
    \end{bmatrix}
    \begin{bmatrix}
    \Phi_{11} & \Phi_{12} \\
    \Phi_{21} & \Phi_{22}
    \end{bmatrix}
    \begin{bmatrix}
    I & X_+ \\ 0 & -X_- 
    \end{bmatrix}^\top > 0.\normalsize\label{e: reduced proj strict}
\end{align}
\end{subequations}
Moreover, if $P>0$ satisfies \eqref{e: reduced schur strict} and \eqref{e: reduced proj strict} then $K$ in \eqref{constructionK} is a stabilizing feedback gain for all systems $(A,B) \in \Sigma$, where $
\Gamma=P -
\bbm
I & X_+
\ebm
\Phi
\bbm
I \\ X^\top_+
\ebm$.
\end{enumerate}
\ethe
\begin{proof}
We first prove \eqref{i:reduction1}. Let $P>0$ be an $n\times n$ matrix and $\beta>0$ be a real number. According to Theorem~\ref{t:theoremstab}, it is enough to show that there exists $L\in\R^{m \times n}$ satisfying \eqref{LMIstab} if and only if \eqref{e: reduced schur} and \eqref{e: reduced proj} are satisfied. By taking the Schur complement of the left hand side in \eqref{LMIstab} with respect to $P$, one can see that there exists $L\in\R^{m \times n}$ satisfying \eqref{LMIstab} if and only if there exists $L\in\R^{m \times n}$ satisfying\vspace{-5pt}
\begin{equation}\label{LMIstab-alt}
\vspace{-5pt}\begin{bmatrix}
    P-\beta I & 0 & 0 \\
    0 & -P & -L^\top  \\
    0 & -L & -LP\inv L^\top
    \end{bmatrix} - \begin{bmatrix}
    I & X_+ \\ 0 & -X_- \\ 0 & -U_- 
    \end{bmatrix}
    \begin{bmatrix}
    \Phi_{11} & \Phi_{12} \\
    \Phi_{21} & \Phi_{22}
    \end{bmatrix}
    \begin{bmatrix}
    I & X_+ \\ 0 & -X_- \\ 0 & -U_- 
    \end{bmatrix}^\top \geq 0.
\end{equation}
By direct inspection, one can verify that \eqref{LMIstab-alt} is equivalent to the following QMI:\vspace{-5pt}
\begin{equation}\label{e:alt qmi}
\vspace{-5pt}\small\bbm
I_n & 0 & 0\\
0 & I_n & 0\\
0 & 0 & I_m\\
0 & 0 & P\inv L^\top 
\ebm^\top
\underbrace{\bbm
\Psi_{11} &\Psi_{12}\\
\Psi_{21} & \Psi_{22}
\ebm}_{:=\Psi}
\bbm
I_n & 0 & 0\\
0 & I_n & 0\\
0 & 0 & I_m\\
0 & 0 & P\inv L^\top
\ebm\geq 0 \normalsize, \textrm{ where }
\end{equation}
\[\Psi_{11}=\small\begin{bmatrix}
    P-\beta I & 0 & 0 \\
    0 & -P & 0 \\
    0 & 0 & 0
    \end{bmatrix}
    -
\begin{bmatrix}
    I & X_+ \\ 0 & -X_- \\ 0 & -U_- 
    \end{bmatrix}
    \begin{bmatrix}
    \Phi_{11} & \Phi_{12} \\
    \Phi_{21} & \Phi_{22}
    \end{bmatrix}
    \begin{bmatrix}
    I & X_+ \\ 0 & -X_- \\ 0 & -U_- 
    \end{bmatrix}^\top\!\!,\,\,\,\normalsize \Psi_{12}=\bbm 0_{n,n} \\ -P \\ 0\ebm, \]
and $\Psi_{22}=-P$. Observe that \eqref{e:alt qmi} is equivalent to \vspace{-5pt}
\begin{equation}\label{e:zero zero K in zset}
\vspace{-5pt}\bbm
0_{n,2n} & P\inv L^\top
\ebm\in\calZ_n(\Psi).
\end{equation}

Now, by defining $W=\bbm I_{2n} & 0_{2n,m}\ebm^\top$ and $Y = 0_{n,2n}$ it follows from Corollary~\ref{cor:elimination} that there exists an $L$ such that \eqref{e:zero zero K in zset} holds if and only if \vspace{-5pt}
\begin{equation}\label{e:condition involving P only}
\vspace{-5pt}\Psi\in\bpi_{2n+m,n}\qand 0_{n,2n}\in\calZ_n(\Psi_W)
\end{equation}
where $\Psi_W$ is defined in \eqref{e:PiW}. Next, we observe that $\Psi\in\bpi_{2n+m,n}$ if and only if $\Psi\schur\Psi_{22}\geq 0$ since $\Psi_{22}=-P<0$. Note that\vspace{-5pt}
\[\vspace{-5pt}\Psi\schur\Psi_{22}=
\small\begin{bmatrix}
    P-\beta I & 0 & 0 \\
    0 & 0 & 0 \\
    0 & 0 & 0
    \end{bmatrix}
    -
\begin{bmatrix}
    I & X_+ \\ 0 & -X_- \\ 0 & -U_- 
    \end{bmatrix}
    \begin{bmatrix}
    \Phi_{11} & \Phi_{12} \\
    \Phi_{21} & \Phi_{22}
    \end{bmatrix}
    \begin{bmatrix}
    I & X_+ \\ 0 & -X_- \\ 0 & -U_- 
    \end{bmatrix}^\top\normalsize. \]
By a Schur complement argument, we see that $\Psi\schur\Psi_{22}\geq 0$ if and only if \eqref{e: reduced schur} holds. Now, observe that $0_{n,2n}\in\calZ_n(\Psi_W)$ if and only if $W^\top\Psi_{11}W\geq 0$. Note that\vspace{-5pt}
\begin{equation}
\label{WPsi11W}
\vspace{-5pt}W^\top\Psi_{11}W=\begin{bmatrix}
    P-\beta I & 0 \\
    0 & -P 
    \end{bmatrix} - \begin{bmatrix}
    I & X_+ \\ 0 & -X_-  
    \end{bmatrix}
    \begin{bmatrix}
    \Phi_{11} & \Phi_{12} \\
    \Phi_{21} & \Phi_{22}
    \end{bmatrix}
    \begin{bmatrix}
    I & X_+ \\ 0 & -X_- 
    \end{bmatrix}^\top.
\end{equation}
Therefore, $0_{n,2n}\in\calZ_n(\Psi_W)$ if and only if \eqref{e: reduced proj} holds. Consequently, the data are informative for quadratic stabilization if and only if there exists an $n\times n$ matrix $P > 0$ and a scalar 
$\beta > 0$ satisfying \eqref{e: reduced schur} and \eqref{e: reduced proj}.

For the construction of the controller, assume that $P>0$ and $\beta>0$ satisfy \eqref{e: reduced schur} and \eqref{e: reduced proj}. By \eqref{WPsi11W}, this implies that $W^\top \Psi_{11} W \geq 0$. This is equivalent to
$- W^\top \Psi_{12} \Psi_{22}^{-1} \Psi_{21} W \leq W^\top (\Psi \schur \Psi_{22}) W$. 

Using Lemma~\ref{l:A A <= B B 2} with $A\! =\! -(-\Psi_{22})^{-\frac{1}{2}} \Psi_{21} W$ and $B = (\Psi\schur \Psi_{22})^\half W$ there exists a matrix $S \in \mathbb{R}^{n \times (2n+m)}$ satisfying $S^\top S \leq I$ and \vspace{-5pt}
\begin{equation}
\label{eqSPsiW}
\vspace{-5pt}-(-\Psi_{22})^{-\frac{1}{2}} \Psi_{21} W = S(\Psi\schur \Psi_{22})^\half W.
\end{equation}
In fact, the matrix $S := -(-\Psi_{22})^{-\frac{1}{2}} \Psi_{21} W \left( (\Psi\schur \Psi_{22})^\half W \right)^\dagger$ works. Since $\Psi_{22} < 0$ and $S^\top S \leq I$, Theorem~\ref{thm:para new} yields $Z := -\Psi_{22}^{-1} \Psi_{21} + (-\Psi_{22})^{-\frac{1}{2}}S(\Psi\schur \Psi_{22})^\half \in \calZ_n(\Psi)$. It follows from \eqref{eqSPsiW} that $ZW = 0$. Now define $K := \begin{bmatrix}
0 & I_m
\end{bmatrix}Z^\top$. Then, by \eqref{e:zero zero K in zset} and Theorem~\ref{t:theoremstab}, $K$ is a stabilizing feedback for all $(A,B) \in \Sigma$. It remains to be shown that $K$ is equal to \eqref{constructionK}. 

First, observe that $Z \begin{bmatrix}
0 \\ I_m
\end{bmatrix} = \Psi_{22}^{-1} \Psi_{21} W \left( (\Psi\schur \Psi_{22})^\half W \right)^\dagger(\Psi\schur \Psi_{22})^\half\begin{bmatrix}
0 \\ I_m
\end{bmatrix}$ because $\Psi_{21} \begin{bmatrix}
0 \\ I_m
\end{bmatrix} = 0$. It can be shown that $\left( (\Psi\schur \Psi_{22})^\half W \right)^\dagger = (\Psi_W \schur \Psi_{22})^\dagger W^\top (\Psi \schur \Psi_{22})^\half$. Then, using the fact that $\Psi_{22}^{-1} \Psi_{21} W = \begin{bmatrix}
0 & I_n
\end{bmatrix}$, this yields\vspace{-5pt}
\[\vspace{-5pt}Z \begin{bmatrix}0 \\ I_m\end{bmatrix} = \begin{bmatrix}0 & I_n\end{bmatrix} (\Psi_W \schur \Psi_{22})^\dagger W^\top (\Psi \schur \Psi_{22}) \begin{bmatrix}0 \\ I_m \end{bmatrix}, \textrm{ which implies }\]
\vspace{-5pt}\begin{equation}
\vspace{-5pt}\label{Kproof}
K = \begin{bmatrix} 0 & I_m \end{bmatrix} (\Psi \schur \Psi_{22}) W (\Psi_W \schur \Psi_{22})^\dagger \begin{bmatrix} 0 \\ I_n \end{bmatrix}.
\end{equation}
Observe that\vspace{-5pt}
\begin{equation}\label{e:K left term}
\vspace{-5pt}\bbm 0 & I_m \ebm
(\Psi\schur\Psi_{22}) W=U_-\bbm \Theta^\top&-\Phi_{22}X_-^\top\ebm.
\end{equation}
Moreover, note that $\Psi_W\schur\Psi_{22}=W^\top(\Psi\schur\Psi_{22})W=
\bbm
\Gamma&\Theta X_-^\top\\
X_-\Theta^\top & \Omega
\ebm$, where $\Omega=-X_-\Phi_{22}X_-^\top$. It follows from \cite[Thm. 2.10]{Tian2005} that\vspace{-5pt}
\begin{equation}\label{e:K right term}
\vspace{-5pt}(\Psi_W\schur\Psi_{22})\gi \bbm 0 \\ I_n\ebm=\bbm \Gamma\gi\Theta X_-^\top\\-I_n\ebm\big(X_-(\Phi_{22}+\Theta^\top\Gamma\gi\Theta)X_-^\top\big)\gi.
\end{equation}
By substituting \eqref{e:K left term} and \eqref{e:K right term} into \eqref{Kproof}, we see that $K$ is equal to \eqref{constructionK}.

The proof of \eqref{i:reduction2} can be established by following the same steps as the proof of \eqref{i:reduction1}, but replaces the term $P-\beta I$ by $P$ and non-strict inequalities by strict ones. Note that in this case, the proof of the if and only if statement relies on Corollary~\ref{cor:strictelimination} rather than Corollary~\ref{cor:elimination} and on Theorem~\ref{t:theoremstab}\eqref{i:infstab2} instead of Theorem~\ref{t:theoremstab}\eqref{i:infstab1}. Also the construction of the controller builds on the results for strict inequalities in Lemma~\ref{l:A A <= B B 2}\ref{item<I} and Theorem~\ref{thm:para new}\eqref{thm:para new.2}, rather than their non-strict counterparts. This proves the theorem.
\end{proof}\vspace{-5pt}

\subsection{Data-driven control of Lur'e systems}

In this section, we will apply some of the results of this paper to the control of Lur'e systems. First, we will explain the classical problem of absolute stability for such systems. Consider the Lur'e system\vspace{-5pt}
\begin{equation}
\vspace{-5pt}\begin{aligned}
\label{Lure}
\bmx(t+1) &= A\bmx(t) + B \bmu(t) + E\varphi(\bmy(t)) \\
\bmy(t) &= C\bmx(t) + D\bmu(t)
\end{aligned}
\end{equation} 
where $\bmx \in \mathbb{R}^n$ is the state, $\bmu \in \mathbb{R}^m$ is the input, $\bmy \in \mathbb{R}$ is the output and  $\varphi : \mathbb{R} \to \mathbb{R}$ is a (nonlinear) function satisfying the so-called \emph{sector condition}\vspace{-5pt}
\begin{equation}
\label{sector}
\vspace{-5pt}\varphi(y) (\varphi(y)-2y) \leq 0 \quad \forall y \in \mathbb{R}.
\end{equation}
We note that more general sector conditions can be transformed to \eqref{sector} by means of loop transformations \cite{Boyd1994}. The real matrices $A, B, C, D$ and $E$ are of appropriate dimensions. Suppose that we apply a state feedback controller $\bmu = K \bmx$ resulting in 
\begin{equation}
\label{Lureclosed}
\begin{aligned}
\bmx(t+1) &= (A+BK)\bmx(t) + E\varphi(\bmy(t)) \\
\bmy(t) &= (C+DK)\bmx(t).
\end{aligned}
\end{equation}
For systems of the form \eqref{Lureclosed}, a problem with a rich history is that of \emph{absolute stability}, i.e. global asymptotic stability of the equilibrium point $0$ \emph{for all} sector-bounded nonlinearities, see \cite{Boyd1994} and the references therein. We focus on showing absolute stability of \eqref{Lureclosed} by means of a quadratic Lyapunov function $V(x) := x^\top P x$ where $P = P^\top > 0$. We thus want that $V(x(t+1)) < V(x(t))$ for all sector-bounded nonlinearities $\varphi$ and all nonzero $x(t)$ and resulting $x(t+1)$ satisfying \eqref{Lureclosed}. We will mimic the continuous-time setting of \cite[Ch. 5]{Boyd1994}. Let $A_K := A+BK$ and $C_K := C+DK$. Then it can be shown \cite{vanWaarde2021} that proving absolute stability of \eqref{Lureclosed} by a quadratic Lyapunov function boils down to finding $P = P^\top > 0$ such that \vspace{-5pt}
\begin{equation}
\label{LMIabsstab}
\begin{bmatrix}
P - A_K^\top P A_K & -A_K^\top P E-C_K^\top \\ -E^\top P A_K-C_K & 1-E^\top P E 
\end{bmatrix} > 0.
\end{equation} 
Next, we focus on data-based stabilization of Lur'e systems. Consider the system \vspace{-5pt}
\begin{equation}
\label{Lure2}
\begin{aligned}
\bmx(t+1) &= A_s\bmx(t) + B_s \bmu(t) + E \varphi(\bmy(t)) +\bmw(t) \\
\bmy(t) &= C_s \bmx(t) + D_s \bmu(t) +\bmv(t),
\end{aligned}
\end{equation}
where $A_s, B_s, C_s$ and $D_s$ are unknown real matrices and the matrix $E$ is known. The signals $\bmw$ and $\bmv$ are process and measurement noise terms that are unknown. We obtain state and input measurements from \eqref{Lure2}, collected in the matrices $X$ and $U_-$ as in Section~\ref{sec:stab}, in addition to corresponding measurements of the form $Y_- = \begin{bmatrix}
y(0) & y(1) & \cdots y(T-1)
\end{bmatrix}$ and $F_- = \begin{bmatrix}
\varphi(y(0)) & \varphi(y(1)) & \cdots & \varphi(y(T-1)) 
\end{bmatrix}$. During the experiment, the noise samples\vspace{-5pt}
\[W_- := \begin{bmatrix}
w(0) & w(1) & \cdots & w(T-1) \\
v(0) & v(1) & \cdots & v(T-1)
\end{bmatrix}\]
are assumed to satisfy $W_-^\top \in \mathcal{Z}_T(\Phi)$ for some known matrix $\Phi \in \bpi_{n+1,T}$. If we define $X_+$ and $X_-$ as before then all systems $(A,B,C,D)$ explaining the data are given by the set $\Sigma$ defined by \vspace{-5pt}
\[\Sigma \!:=\! \set{(A,B,C,D)}{ \begin{bmatrix} X_+-EF_- \\ Y_- \end{bmatrix} \!-\! \begin{bmatrix}
A & B \\ C & D
\end{bmatrix} \begin{bmatrix}
X_- \\ U_-
\end{bmatrix} = W_- \:\: \text{ for some } W_-^\top \in \mathcal{Z}_T(\Phi)}.\]
This leads to the following definition of informative data. 
\begin{definition}
Let $\Phi \in \bpi_{n+1,T}$. Suppose that the data $(U_-,F_-,X,Y_-)$ have been generated by \eqref{Lure2} for some noise sequence $W_-^\top \in \mathcal{Z}_T(\Phi)$. Then $(U_-,F_-,X,Y_-)$ are called \emph{informative for absolute quadratic stabilization} if there exists an $n \times n$ matrix $P > 0$ and a $K \in \mathbb{R}^{m \times n}$ such that \eqref{LMIabsstab} holds for all $(A,B,C,D) \in \Sigma$. 
\end{definition}
Next, we state the following theorem that gives a necessary and sufficient condition for informativity for absolute quadratic stabilization. 
\begin{theorem}\label{thm:lure}
Let $\Phi \in \bpi_{n+1,T}$ and consider the data $(U_-,F_-,X,Y_-)$, generated by \eqref{Lure2} for some $W_-^\top \in \mathcal{Z}_T(\Phi)$. Then $(U_-,F_-,X,Y_-)$ are informative for absolute quadratic stabilization if and only if there exists an $n \times n$ matrix $Q > 0$, an $L \in \mathbb{R}^{m\times n}$ and scalars $\alpha \in \mathbb{R}$ and $\beta > 0$ such that 
\vspace{-10pt}\begin{equation*}
\small\begin{bmatrix}
    Q-\beta I & -E & 0 & 0 & 0 \\
    -E^\top & 1-\beta & 0 & 0 & 0 \\
    0 & 0 & 0 & 0 & Q \\
    0 & 0 & 0 & 0 & L \\
    0 & 0 & Q & L^\top & Q 
    \end{bmatrix}
     +\alpha\! \begin{bmatrix}
I & 0 & X_+-EF_- \\ 0 & 1 & Y_- \\ 0 & 0 & -X_- \\ 0 & 0 & -U_- \\ 0 & 0 & 0
\end{bmatrix}
\!\Phi \!
\begin{bmatrix}
I & 0 & X_+-EF_- \\ 0 & 1 & Y_- \\ 0 & 0 & -X_- \\ 0 & 0 & -U_- \\ 0 & 0 & 0
\end{bmatrix}^\top \!\! \geq  0\normalsize. \label{LMIabsstabdata}
\end{equation*}
In this case, $K := LQ^{-1}$ is such that \eqref{Lureclosed} is absolutely stable for all $(A,B,C,D) \in \Sigma$. 
\end{theorem}
The proof follows similar lines as that of Theorem~\ref{t:theoremstab}. It uses a dualization step \cite[Lem. 4.9]{Scherer1999} on the inequality \eqref{LMIabsstab} and relies on Corollary~\ref{c:combinedstrictS-lemmaFinslerslemma} using the relation $Q = P^{-1}$ between $P$ and $Q$.

	\subsection{The setting of stochastic noise}
	\label{sec:stochastic}
	
	Consider an unknown system\vspace{-5pt}
	\[\vspace{-5pt}\bmx(t+1) = A_s\bmx(t) + B_s \bmu(t) + \bmw(t), \]
	where $\bmx(t), \bmw(t)\in\mathbb{R}^n$, $\bmu(t)\in\mathbb{R}^m$, and the matrices $A_s,B_s$ have the appropriate dimensions. Assume that $\bmw \sim \calN(0,\sigma^2I)$, that is, $\bmw$ is normally distributed with mean $0$ and covariance $\sigma^2I$. We suppose that $\sigma$ is known and that we have access to a finite number of state and input measurements. Clearly, any linear system of the given dimensions $n$ and $m$ could have generated these measurements, yet not every system is a probable explanation. We are interested in characterizing sets of systems that contain the true system with a given probability. For this, we will follow the method of \cite{Umenberger2019}.
	
	Let $X$, $U_-$, $X_-$, and $X_+$ be as before. Clearly, if no noise were present we would have $X_+ = A_sX_- + B_sU_-$, and the true system $(A_s,B_s)$ would be a solution to this inhomogeneous linear equation. However, due to the noise term, we consider its least squares solutions.
	Let $\Vert\cdot\Vert_F$ denote the Frobenius norm of a given matrix and let $(\hat{A},\hat{B})$ be an ordinary least squares estimate of the matrix pair $(A_s,B_s)$, that is, \vspace{-5pt}
	\[ (\hat{A},\hat{B}) \in \argmin_{(A,B)\in\mathbb{R}^{n\times n+m}} \Big\Vert X_+ - \begin{bmatrix} A & B \end{bmatrix}\begin{bmatrix} X_- \\ U_-\end{bmatrix}\Big\Vert_F . \] 
	We can employ the Moore-Penrose inverse to explicitly write a solution of the previous as $\begin{bmatrix} \hat{A} &\hat{B}\end{bmatrix} = X_+ \begin{bmatrix} X_- \\ U_-\end{bmatrix}^\dagger$. This solution is {\em unique} if and only if the matrix $\begin{bmatrix} X_- \\ U_-\end{bmatrix}$ has full row rank $n + m$. Assume this to be the case in the sequel.
	
	
Denote the $\chi$-squared probability distribution with $k$ degrees of freedom by $\chi^2_k$. The paper \cite{Umenberger2019} shows that the vectorization of $\begin{bmatrix} A & B\end{bmatrix}$ is distributed according to $\chi^2_{n(n+m)}$, with as mean the vectorization of the least square estimate $\begin{bmatrix} \hat{A} &\hat{B}\end{bmatrix}$. Using a generalization of the aforementioned confidence intervals for vector variables, we can then formulate confidence intervals for the matrix variable $(A,B)$. Given probability $0<\delta<1$, denote the quantile function for probability $\delta$ of $\chi^2_{n(n+m)}$ by $c_\delta$. Now, define the set\vspace{-5pt}
\[\vspace{-5pt}\Theta_\delta :=\left\lbrace (A,B) \mid \begin{bmatrix} \hat{A}^\top-A^\top \\ \hat{B}^\top-B^\top\end{bmatrix}^\top \begin{bmatrix} X_- \\ U_-\end{bmatrix}\begin{bmatrix} X_- \\ U_-\end{bmatrix}^\top \begin{bmatrix} \hat{A}^\top-A^\top \\ \hat{B}^\top-B^\top\end{bmatrix} \leq \sigma^2c_\delta I \right\rbrace.\]
	
The following lemma was proven in \cite[Lem. 3.1]{Umenberger2019}.
	
\begin{lemma}\label{lem: stoch} \vspace{-3pt}Let $0<\delta<1$. Then with probability $1-\delta$ we have that $(A_s, B_s) \in \Theta_\delta$.
\end{lemma}
	
	\vspace{-3pt}Recall that we are interested in finding a stabilizing controller for the system $(A_s,B_s)$. Suppose that there exists a gain $K$ and a matrix $P>0$ such that the Lyapunov inequality \eqref{lyapunovineq} holds for all $(A,B)\in\Theta_\delta$. Then we can conclude from Lemma~\ref{lem: stoch} that, with probability $1-\delta$, this gain also stabilizes the true system $(A_s,B_s)$. This leads to the following definition:

	
	\begin{definition}\label{def:stoch}
		\vspace{-5pt}Given $0<\delta<1$, the data $(U_-,X)$ are called \emph{informative for quadratic stabilization with probability $1-\delta$} if there exists a feedback gain $K$ and a matrix $P > 0$ such that the Lyapunov inequality \eqref{lyapunovineq} holds for all $(A,B)\in\Theta_\delta$. 
	\end{definition}
	
	It should be stressed that Definition~\ref{def:stoch} is dependent on the construction of $\Theta_\delta$. However, the choice of confidence interval is not unique. In general, there are different sets $\bar{\Theta}$ that also contain $(A_s,B_s)$ with probability $1-\delta$. Each of such sets could lead to a definition analogous to Definition~\ref{def:stoch}. In order to apply the results of this paper, we choose to base our definition on the specific choice of $\Theta_\delta$.
	In the following theorem we will apply the framework of this paper to provide necessary and sufficient conditions for this notion of informativity. \vspace{-5pt}
	
	\begin{theorem}\label{t:stoch}
		Given $0<\delta<1$, and data $(U_-,X)$ such that $\begin{bmatrix} X_- \\ U_-\end{bmatrix}$ has full row rank. Then the data $(U_-,X)$ are informative for quadratic stabilization with probability $1-\delta$ if and only if there exists an $n\times n$ matrix $P > 0$, an $L \in \mathbb{R}^{m \times n}$ and a scalar $\beta > 0$ satisfying \vspace{-5pt}
		\begin{equation}
		\vspace{-5pt}\small\begin{bmatrix}
				P-\beta I\!\!\!\! & 0 & 0 & 0 \\
				0 & \!\!-P & \!\!-L^\top & 0 \\
				0 & \!\!-L & 0 & L \\
				0 & 0 & L^\top & P \\
				\end {bmatrix}\!\! - \!\!
				\begin{bmatrix} I & X_+ \\ 0 & \!-X_- \\ 0 & \!-U_- \\ 0 &0 \end{bmatrix} \begin{bmatrix} 	\sigma^2c_\delta I &0 \\ 0& -\begin{bmatrix} X_- \\ U_-\end{bmatrix}^\dagger\begin{bmatrix} X_- \\ U_-\end{bmatrix}\end{bmatrix} \begin{bmatrix} I & X_+ \\ 0 & \!-X_- \\ 0 & \!-U_- \\ 0 &0 \end{bmatrix}^\top \!\!\geq 0.
				\normalsize\label{LMIstabstoch}
			\end{equation}
			Moreover, if $P$ and $L$ satisfy \eqref{LMIstabstoch} then $K := L P^{-1}$ is a stabilizing feedback gain for $(A_s,B_s)$ with probability at least $1-\delta$.
		\end{theorem}
		
		\begin{proof} 
		\vspace{-5pt}Let $M$ be defined as in \eqref{defM}. Clearly, the data is informative for quadratic stabilization with probability $1-\delta$ if and only if there exists $P>0$ and $K$ such that $\Theta_\delta \subseteq \calZ_{n+m}^+(M)$. The proof of this theorem will continue as follows. First, we will define a matrix $\Phi\in\bm{\Pi}_{n,T}$. After this, we show that by defining $N$ as in \eqref{defN}, we obtain $\Theta_\delta=Z_{n+m}(N)$. Using this, we will prove the theorem by invoking Theorem~\ref{t:theoremstab}. Let\vspace{-5pt}
		\[ \vspace{-5pt}\Phi := \small\begin{bmatrix}\sigma^2c_\delta I &0 \\ 0& -\begin{bmatrix} X_- \\ U_-\end{bmatrix}^\dagger\begin{bmatrix} X_- \\ U_-\end{bmatrix}\end{bmatrix}\normalsize. \]  
		Since $\begin{bmatrix} X_- \\ U_-\end{bmatrix}$ has full row rank, we see that $\Phi_{22}\leq 0$. Moreover, we can immediately see that $\Phi\schur \Phi_{22} = \sigma^2c_\delta I>0$. We may conclude that $\Phi\in\bm{\Pi}_{n,T}$. Let $N$ be defined as in \eqref{defN}.
			By performing some routine operations regarding the pseudo-inverse, and employing
			the definition of $(\hat{A},\hat{B})$, we can conclude that\vspace{-5pt}
			\[\vspace{-5pt}N= \small\begin{bmatrix} \sigma^2c_\delta I - 
				\begin{bmatrix} \hat{A}^\top \\ \hat{B}^\top \end{bmatrix}^\top\begin{bmatrix} X_- \\ U_-\end{bmatrix}\begin{bmatrix}X_- \\ U_-\end{bmatrix}^\top
				\begin{bmatrix} \hat{A}^\top \\ \hat{B}^\top \end{bmatrix}  & 
				\begin{bmatrix} \hat{A}^\top \\ \hat{B}^\top \end{bmatrix}^\top\begin{bmatrix} X_- \\ U_-\end{bmatrix}\begin{bmatrix}X_- \\ U_-\end{bmatrix}^\top \\\begin{bmatrix} X_- \\ U_-\end{bmatrix}\begin{bmatrix}X_- \\ U_-\end{bmatrix}^\top
				\begin{bmatrix} \hat{A}^\top \\ \hat{B}^\top \end{bmatrix} & -\begin{bmatrix} X_- \\ U_-\end{bmatrix}\begin{bmatrix} X_- \\ U_-\end{bmatrix}^\top \end{bmatrix}. \normalsize
			\]
			
			As such, it can be concluded that $\Theta_\delta = Z_{n+m}(N)$. This shows that informativity for quadratic stabilization with probability $1-\delta$ is equivalent to informativity for quadratic stabilization corresponding to the specific noise model given by $\Phi$. We can now invoke Theorem~\ref{t:theoremstab} to finalize the proof.
		\end{proof}
\section{Conclusions}
\label{s:conclusions}

In this paper we have studied properties of quadratic matrix inequalities. We have established conditions under which the solution sets of such QMI's are nonempty, convex, bounded, or have nonempty interior (Theorem~\ref{t:basic props}). In addition, we have given parameterizations of the solution set of a given QMI, for both strict and nonstrict inequalities (Theorem~\ref{thm:para new}). We have also shown that under suitable conditions all matrices of the form $ZW$, where $Z$ is a solution to a QMI and $W$ is a given matrix, are again solutions to a (different) QMI (Theorems~\ref{t:projectionfcr} and \ref{t:projectionfcrstrict}). Finally, we have established matrix versions of the classical S-lemma and Finsler's lemma, that provide LMI conditions under which all solutions to one QMI also satisfy another QMI (Theorems~\ref{t:nonstrictS-lemma}, \ref{t:matFinsler}, \ref{t:strictS-lemmaN22}, \ref{t:strictS-lemma} and \ref{t:strictmatFinsler}). We have demonstrated that our results generalize previous work, and in particular, do not require the assumption of the generalized Slater condition \cite{vanWaarde2022}. In the case of a strict inequality, we have given a combined matrix S-lemma and Finsler's lemma, that does not even require the ``ordinary" Slater condition (Corollary~\ref{c:combinedstrictS-lemmaFinslerslemma}). It was also shown that both the strict and the nonstrict Petersen's lemma \cite{Petersen1986,Petersen1987} can be obtained from the matrix S-lemmas of this paper (Proposition~\ref{p:Petersenslemma}). 

We have studied applications of the various QMI results in the context of data-driven control. In particular, we have studied the problem of finding a stabilizing controller of an unknown LTI system influenced by noise from a finite set of input-state samples. To this end, we have assumed that the (unknown) matrix of noise samples satisfies a QMI. This led to a set of solutions to a QMI consisting of LTI systems that ``explain" the data. We have applied the combined matrix S-lemma and Finsler's lemma to establish a necessary and sufficient LMI condition under which all systems explaining the data can be quadratically stabilized by a single controller (Theorem~\ref{t:theoremstab}). Subsequently, we have reduced the computational complexity of this scheme by separating the computation of the Lyapunov function and the controller, leading to lower-dimensional LMI's (Theorem~\ref{t:reducedLMIs}). This result also included an explicit formula for a stabilizing controller, given a (pre-computable) Lyapunov function. We then tackled data-based stabilization of a class of Lur'e systems in Theorem~\ref{thm:lure}. Lastly, we applied the methods of this paper in order to derive Theorem~\ref{t:stoch}, which provides probabilistic stabilization guarantees in the presence of Gaussian noise.

So far, we have only applied the QMI results to data-driven stabilization. Nonetheless, we believe that our results are also relevant for other data-based analysis and control problems that involve QMI's. Some examples of this are $H_{2}$ and $H_{\infty}$ control \cite{Berberich2019c,Steentjes2021}, data-driven dissipativity analysis \cite{Koch2021}, as well as data-driven reduced order modeling \cite{Burohman2021}. In particular, the more general matrix S-lemmas of this paper will lead to weaker conditions on the data for these analysis and design problems. Another topic for future research is to develop dedicated algorithms for solving the various linear matrix inequalities in this paper.

\appendix

\section{Facts from matrix theory}\label{s:appendix} We start with the following:

\begin{appxlem}\label{l:A A <= B B 2}
Let $A\in\mathbb{R}^{r\times q}$ and $B\in\mathbb{R}^{p\times q}$. 
\begin{enumerate}[label=(\alph*),ref=(\alph*)]
\item\label{item<=I} $A^\top A\leq B^\top B$ if and only if there exists $S\in\R^{r\times p}$ such that \vspace{-5pt}
\begin{equation}\label{eqsS}
\vspace{-5pt}A=SB \text{ and } S^\top S\leq I.
\end{equation} 
\item\label{item<I} Assume, in addition, that $B$ has full column rank. Then $A^\top A < B^\top B$ if and only if there exists $S\in\R^{r\times p}$ such that \vspace{-5pt}
\begin{equation}\label{eqsSstrict}
\vspace{-5pt}A=SB \text{ and } S^\top S < I.
\end{equation}
\end{enumerate}
Moreover, if $A^\top A - B^\top B \leq 0$ (respectively, $< 0$), then $S := AB^\dagger$ satisfies \eqref{eqsS} (respectively, \eqref{eqsSstrict}).
\end{appxlem}

The proof of Lemma~\ref{l:A A <= B B 2} for the case $r=q=p$ is given in \cite[Fact 5.10.19]{Bernstein2009} and for the case $r=p$ in \cite[Lem. 3]{Rantzer1996}. Here we provide a constructive proof for the case that $r$ and $p$ are not necessarily equal.

\begin{proof} 

To prove the `if' parts of statements \ref{item<=I} and \ref{item<I}, assume that $A = SB$ with $S^\top S \leq I$ (respectively, $< I$). Then $A^\top A = B^\top S^\top SB \leq B^\top B$ (respectively, $< B^\top B$), where we have made use of full column rank of $B$ to prove the strict inequality.

Next, we prove the `only if' part of \ref{item<=I}. We thus assume that $A^\top A \leq B^\top B$. Our goal is to show that $S:=AB^\dagger$ satisfies \eqref{eqsS}. First, note that $A^\top A \leq B^\top B$ implies that $\ker B \subseteq \ker A$, equivalently, $\im A^\top \subseteq \im B^\top$. Thus, there exists a matrix $Z \in \mathbb{R}^{r \times q}$ such that $A^\top = B^\top Z^\top$, equivalently, $A = Z B$. Therefore, $SB = A B^\dagger B = Z B B^\dagger B = Z B = A$. Moreover, $S^\top S = (B^\dagger)^\top A^\top A B^\dagger \leq (B^\dagger)^\top B^\top B B^\dagger = B B^\dagger \leq I$, where the last inequality follows from the fact that $BB^\dagger$ is an orthogonal projection matrix. This shows that $S$ satisfies \eqref{eqsS}. To prove the `only if' part of statement \ref{item<I}, assume that $B$ has full column rank and $A^\top A < B^\top B$. This implies that there exists an $\epsilon > 0$ such that $(1+\epsilon) A^\top A \leq B^\top B$. As such, by statement \ref{item<=I}, the matrix $\bar{S}:= \sqrt{1+\epsilon}AB^\dagger$ satisfies $\sqrt{1+\epsilon} A = \bar{S} B$ and $\bar{S}^\top \bar{S} \leq I$. Define $S := \frac{1}{\sqrt{1+\epsilon}} \bar{S} = AB^\dagger$. Then $A = SB$ and $S^\top S = \frac{1}{1+\epsilon}\bar{S}^\top \bar{S} < I$. We conclude that $S$ satisfies \eqref{eqsSstrict}.
\end{proof}

The following is a direct consequence of \cite[Prop. 6.1.7]{Bernstein2009}.
\begin{appxlem}\label{lem: image-exist} Let $A\in\mathbb{R}^{q\times p}$ and $B\in\mathbb{R}^{r\times p}$. Then $AM=B$ if and only if $\im B \subseteq\im A$ and there exists $T\in\mathbb{R}^{r\times q}$ such that $M = A\gi B+ (I_q-A\gi A)T$. 
\end{appxlem}

\bibliographystyle{siamplain}
\bibliography{references}
\end{document}